\documentclass[12pt]{amsart}

\usepackage{titletoc}
\usepackage{hyperref}
\usepackage{verbatim}
\usepackage{setspace}
\usepackage{graphics}

\usepackage{amsmath}
\usepackage{amsthm}
\usepackage{amsfonts}
\usepackage{stmaryrd}
\usepackage{amssymb}
\usepackage[all]{xy}
\usepackage{comment}
\usepackage{graphicx}

\numberwithin{equation}{section}
\newcommand\R{\mathbb{R}}

\newcommand{\oD}{ \!{\buildrel \circ 
\over D}}

\newcommand{\D}{\mathcal{D}}
\newcommand{\T}{\mathcal{T}}

\newcommand{\kRP}{\mathcal{R}^{Ric_k>0}}
\newcommand{\RP}{\mathcal{R}^{Ric>0}}
\newcommand{\rd}{\mathcal{R}^{Ric>0}_{rd}}
\newcommand{\rdo}{\mathcal{R}^{Ric>0}_{rd,1}}

\newcommand{\rdk}{\mathcal{R}^{Ric_k>0}_{rd}}
\newcommand{\rdok}{\mathcal{R}^{Ric_k>0}_{rd,1}}

\newcommand{\bdy}{\mathcal{R}^{Ric_k>0}(\bar{M})^*}

\newcommand{\In}{\mathrm{In}}

\DeclareMathOperator{\Ric}{Ric}
\DeclareMathOperator{\Rk}{Ric_k}

\theoremstyle{plain}
\newtheorem{theorem}{Theorem}[section]
\newtheorem{lemma}[theorem]{Lemma}
\newtheorem{cor}[theorem]{Corollary}
\newtheorem{proposition}[theorem]{Proposition}

\theoremstyle{definition}
\newtheorem{definition}[theorem]{Definition}

\newtheorem{obs}[theorem]{Observation}

\theoremstyle{remark}
\newtheorem{remark}[theorem]{Remark}
\newtheorem{question}[theorem]{Question}

\usepackage{color}

\begin{document}

     
     
     

\title{H-space and loop space structures for intermediate curvatures}
      
\author{Mark Walsh}
\address{
Department of Mathematics and Statistics\\
Maynooth University \\
Maynooth \\
Ireland
}
\email{mark.walsh@mu.ie, david.wraith@mu.ie}

\author{David J.~Wraith}


\maketitle

\begin{abstract} For dimensions $n\geq 3$ and $k\in\{2, \cdots, n\}$, we show that the space of metrics of $k$-positive Ricci curvature on the sphere $S^{n}$ has the structure of an $H$-space with a homotopy commutative, homotopy associative product operation. We further show, using the theory of operads and results of Boardman, Vogt and May that the path component of this space containing the round metric is weakly homotopy equivalent to an $n$-fold loop space.
\end{abstract}

\section{Introduction}\label{intro}

There has been a great deal of interest in recent years about the topology of the space of Riemannian metrics sastisfying given curvature conditions on a given manifold. (As a starting point for this topic, see \cite{TW}.) Interest has been mainly directed towards studying the homotopy and (co)homology groups of these spaces of metrics, with many results demonstrating that these algebraic invariants are often non-trivial. There are, of course, other aspects of topology which are not captured by computing homotopy and homology. In this paper, we focus on the existence of $H$-space structures and loop space structures in certain spaces of metrics.

Recall (see for example \cite[page 224]{Maunder}) that a topological space $X$ is said to be an $H$-space if it has a `multiplication' map $m:X \times X \to X$, with an identity element $e \in X$ such that $m\circ \iota_1 \simeq m\circ \iota_2\simeq \text{id}_X,$ where $\iota_1:X \to X \times X$ is the map $\iota_1(x)=(x,e),$ and $\iota_2(x)=(e,x).$ We will say that an $H$-space is (homotopy) associative if $m(m \times \text{id}_X) \simeq m(\text{id}_X \times m),$ and (homotopy) commutative if $m \simeq m\circ S,$ where $S$ is the `swapping' map $S:X \times X \to X$ given by $S(x_1,x_2)=(x_2,x_1).$

Recall also that $X$ is said to be a loop space (see \cite[page 216]{Maunder}) if there exists a based topological space $(Y,y_0)$ for which $X=\Omega Y,$ where $\Omega Y$ is the set of loops in $Y$ based at $y_0$, equipped with the compact-open topology. In this paper we will be particularly interested in iterated loop spaces, i.e. $\Omega^n Y=\Omega(\Omega(\cdots (\Omega Y) \cdots).$ Note that a loop space is always an $H$-space, as a multiplication map can be constructed by concatenating loops. On the other hand, the question of whether an $H$-space is a loop space is highly non-trivial in general. We will return to this point later.

The motivation behind the results in this paper was the work of the first author in \cite{Wa}. This paper studied the space of positive scalar curvature metrics on the sphere $S^n$, and demonstrated the existence of an $H$-space structure on this space of metrics whenever $n\ge 3.$ It was also shown that the path-component of the round metric admits an $n$-fold loop space structure. The current paper arose from exploring the extent to which these positive scalar curvature results continue to hold for stronger curvature conditions.

Before going any further, we must mention the fact that \cite{Wa} is not the only paper in the literature which studies $H$-space or loop space structures for positive scalar curvature metrics. In \cite{ERW}, Ebert and Randall-Williams prove something stronger than the main result in \cite{Wa}, namely that a certain union of path components in the space of metrics of positive scalar curvature on $S^{n}$ (containing the path component of the round metric) is homotopy equivalent to an infinite loop space (at least when $n\geq 6$). The methods used here are heavily homotopy theoretic and very different to those used in \cite{Wa}. As the authors of \cite{ERW} point out, it is difficult to compare these constructions and it is unclear as to whether or not the structures built in \cite{ERW} extend those of \cite{Wa}. More recently still, in \cite{Frenck}, G. Frenck has demonstrated that for any compact spin manifold of dimension at least six, the resulting space of positive scalar curvature metrics admits a homotopy-commutative, homotopy-associative H-space structure. 

Although the classical scalar, Ricci, and sectional curvatures have always been of central importance in Riemannian geometry, there is currently an increasing interest in more subtle notions of curvature. For example, there is the notion of $p$-curvature, which interpolates between positive scalar (when $p=0$) and positive sectional curvature (when $p=n-2$), (see for example \cite{La}, \cite{BL1}). Then there is the notion of $k$-positive curvature (also known in the literature as the $k^{th}$-intermediate Ricci curvature or $k^{th}$-Ricci curvature), see for example  \cite{Hart}, \cite{Sh1}, \cite{Sh2}, \cite{Wu}, \cite{Wi}, \cite{GX}, \cite{GW1}, \cite{GW2}, \cite{GW3} \cite{Mo1}, \cite{Mo2}. We take special note of the paper \cite{Kor}, which concerns the topology of the space of metrics satisfying so-called `surgery stable' curvature conditions, building on ideas developed in \cite{Ho}. This includes positive scalar curvature, as well as a number of other conditions. To the best of the authors' knowledge, this paper is the first  to study spaces of metrics satisfying non-classical curvature conditions. For example, Corollaries D and E of that paper are results about the homotopy type of the space of metrics with positive $p$-curvature in the case $p=1$.  

In this paper we will focus on a curvature condition introduced by Wolfson in \cite{Wo}:
\begin{definition} We say that an $n$-dimensional Riemannian manifold has $k$-positive Ricci curvature if the sum of the $k$ smallest eigenvalues of the Ricci tensor is positive at all points. We will write this as $Ric_k>0.$
\end{definition}
Notice that $n$-positive Ricci curvature is just positive scalar curvature, and 1-positive Ricci curvature is the same as positive Ricci curvature. Thus the $k$-positive Ricci curvatures provide a very natural family of curvatures intermediate between positive scalar and positive Ricci curvature. We will denote by $\kRP(M)$ the space of all $k$-positive Ricci metrics on $M$ (equipped with the smooth topology).

The main results established in this paper are as follows:
\medskip

\noindent {\bf Theorem A.} {\em When $n \geq 3$ and $2 \le k\le n$, $\kRP(S^{n})$ has the structure of an $H$-space with a homotopy commutative, homotopy associative product.}
\\

\noindent {\bf Corollary B.} {\em When $n \geq 3$ and $2 \le k\le n$, the fundamental group of $\kRP(S^{n})$, with basepoint the standard round metric, is abelian.}
\\

\noindent {\bf Theorem C.} {\em When $n \geq 3$ and $2 \le k\le n$, the path component of $\kRP(S^{n})$ containing the round metric is weakly homotopy equivalent to an $n$-fold loop space.}
\medskip

A key feature of $k$-positive Ricci curvature is that under a certain codimension condition, it can be preserved by performing surgeries. The result here is as follows.
\begin{theorem}[\cite{Wo},\cite{Ho}] \label{thm:wolfson} 

Let $M^n$ be a closed Riemannian manifold with $k$-positive Ricci curvature, $2 \le k \le n$. Then any manifold obtained from $M$ by performing surgeries in codimension $q$ with $q \ge \max\{n+2-k,3\}$ also admits a metric of $k$-positive Ricci curvature. In particular if $M_1$ and $M_2$ are manifolds of dimension $n \ge 3$ which admit metrics of 2-positive Ricci curvature, then the connected sum $M_1 \sharp M_2$ also admits a metric of 2-positive Ricci curvature.
\end{theorem}

It is the fact that $k$-positive Ricci curvature, $2 \le k \le n$, is preseved under connected sums which is crucial for the results in this paper. Naively, one might consider two $k$-positive Ricci metrics on $S^n$. We can join these by a connected sum within $Ric_k>0,$ to give a new metric on a sphere. (For technical reasons, it turns out to be better to perform connected sums between each of the spheres and a round sphere of fixed radius, a so-called `docking station'.) The problem with this construction is that the original metrics were metrics on a standard smooth sphere, and to exhibit a well-defined $H$-space multiplication, our final metric must also be a metric on the standard sphere. Thus we must find a diffeomorphism between the standard sphere and the connected sum arrangement we construct, with which to pull back the metric. Of course such diffeomorphisms exist in abundance, but in order to have a well-defined multiplication, we must show that such a diffeomorphism can be chosen in a standard way, depending smoothly on the individual metrics involved. (We call this the `connected sum contraction procedure'.) 

In Section 2, starting from a result about deforming positive Ricci curvature metrics \cite{Wr}, we deduce a local rounding result for $k$-positive Ricci curvature, $1 \le k \le n$. This is useful for the constructing and manipulating the connected sum constructions mentioned above: in practice it means that we only have to deal with warped product metrics. It also means that for $H$-space considerations everything can be controlled in terms of a single parameter $R$ arising from each metric, (and for loop space considerations two parameters $R$ and $\epsilon$).

Given that working with metrics which are locally round is convenient, in Section 3 we prove that a certain space of such metrics on the sphere has an $H$-space structure. This space has the same homotopy type as the full space of $k$-positive Ricci metrics, and so this is enough to prove Theorem A: it is an easy exercise (see for example \cite[page 251]{Maunder}) to show that if $X$ is an $H$-space and $X \simeq Y,$ then $Y$ is also an $H$-space.  The existence of the homotopy identity turns out to be particularly delicate, and for this we develop a `warped product deformation procedure', which turns out to be equally useful in our loop space considerations.


On a more technical level, since we are concerned with defining smooth operations on spaces of metrics, we have to work to remove all choices from our constructions so that the operations are unambiguously defined, and can be seen to vary smoothly with the input metrics. For this reason we take a different approach to that of Wolfson when constructing connected sums: it turns out to be convenient to construct the main building block for the tube in a single piece. (The Wolfson approach is based on the classic Gromov-Lawson construction \cite{GL}.) On the other hand, this requires some smoothing at the ends. Indeed the `warped product deformation procedure' mentioned above also requires a smoothing argument. We therefore have to take speical care of how we smooth: in part to remove choices, and in part to make sure that the smoothings themselves vary smoothly with the input metrics.

The final section of the paper, Section 4, is dedicated to proving Theorem C. Our approach here mimics that of \cite{Wa} for positive scalar curvature. The key to establishing (iterated) loop space structure is provided by loop space recognition results due to Boardman and Vogt \cite{BV}, and May \cite{May}. These results rely on the concept of an operad, and in particular operad actions on topological spaces. This is explained briefly in Section 4, but see \cite{Wa} and the references therein for a more detailed exposition. In short, to establish Theorem C, it suffices to exhibit an action of a certain operad on the space of $k$-positive Ricci metrics on $S^n$. The construction we perform here relies heavily on the techniques developed in Section 3.

In conclusion, one might speculate about the existence of $H$-space or loop space structures for other spaces of metrics. In particular, since the ability to do connected sums within $k$-positive Ricci curvature for $2 \le k \le n$ is a crucial feature behind the results in this paper, it is natural to ask whether similar arguments can be made for other curvature conditions which allow connected sums. We conclude, however, with the following question:
\begin{question}
Does the space of Ricci positive metrics on $S^n$ admit an $H$-space or (iterated) loop space structure?
\end{question}
\noindent It should be noted that the techniques used in this paper appear to offer no insight into this question.


\section{Spaces of $k$-positive Ricci metrics}

Suppose that $M^n$ is a closed manifold which supports a metric with $\Rk>0$ for some $k\in\{1, \cdots, n\}$. Our aim in this section is to define and analyse certain spaces of metrics on $M$. Recall that $\kRP(M)$ denotes the space of all $k$-positive Ricci metrics on $M$. Now choose a basepoint $x_0 \in M$, and for reasons that will become clear later, we will fix a basis $\{e_1,...,e_n\}$ for the tangent space $T_{x_0}M$. We will denote by $\rdk(M)$ the space of $k$-positive Ricci metrics on $M$ for which the restriction to some neighbourhood of $x_0$ is round. Define $\rdok(M)$ to be the subspace of $\rdk$ for which the round neighbourhood contains a distance sphere about the basepoint $x_0$ on which the induced metric has radius 1. 
If $M$ actually supports a Ricci positive metric (i.e. a 1-positive Ricci metric) then we will simply write $\RP(M),$ $\rd(M)$ and $\rdo(M)$ respectively for these spaces of metrics. Notice that for any metric in $\rdok,$ we can, and will, assume that removing the open ball about $x_0$ bounded by the unit (intrinsic) radius distance sphere leaves a concave boundary.
In summary, we have specified spaces which include as follows:
$$ \rdok(M)\subset\rdk(M)\subset \kRP(M).$$

\begin{figure}[!htbp]
\vspace{1cm}
\hspace{0.5cm}
\begin{picture}(0,0)
\includegraphics{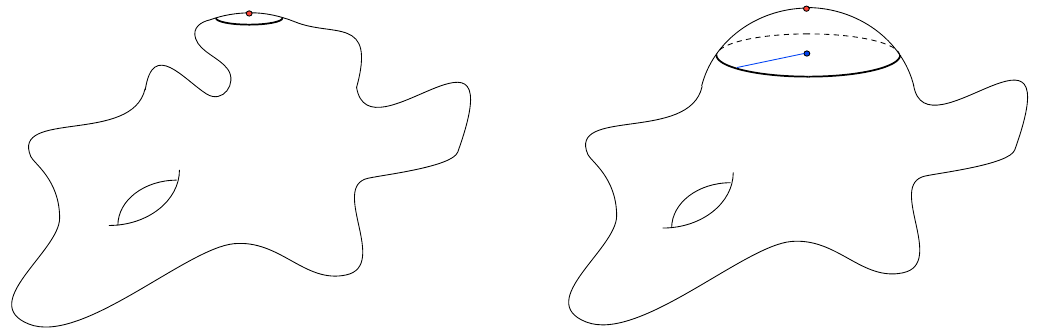}%
\end{picture}
\setlength{\unitlength}{3947sp}
\begingroup\makeatletter\ifx\SetFigFont\undefined%
\gdef\SetFigFont#1#2#3#4#5{%
  \reset@font\fontsize{#1}{#2pt}%
  \fontfamily{#3}\fontseries{#4}\fontshape{#5}
  \selectfont}%
\fi\endgroup%
\begin{picture}(5079,1559)(1902,-7227)
\put(2900,-5600){\makebox(0,0)[lb]{\smash{{\SetFigFont{10}{8}{\rmdefault}{\mddefault}{\updefault}{\color[rgb]{0,0,0}$x_0$}%
}}}}
\put(5550,-5600){\makebox(0,0)[lb]{\smash{{\SetFigFont{10}{8}{\rmdefault}{\mddefault}{\updefault}{\color[rgb]{0,0,0}$x_0$}%
}}}}
\put(5450,-5900){\makebox(0,0)[lb]{\smash{{\SetFigFont{10}{8}{\rmdefault}{\mddefault}{\updefault}{\color[rgb]{0,0,1}$1$}
}}}}
\end{picture}%
\caption{Sample elements of $\rd(M)$ (left) and $\rdo(M)$ (right)}
\label{SampleMetrics}
\end{figure}

Our starting point for the analysis of these spaces of metrics is a result of the second author \cite{Wr}, which generalizes to arbitrary submanifolds in a Ricci positive setting a result for curves in negative Ricci curvature due to Gao \cite{G}.
\begin{theorem}
\label{David2002}{\rm (\cite{Wr})}
Let $X$ be a manifold and $Y$ a compact submanifold with $\dim Y< \dim X$.  Let $g_1$ be a Ricci positive metric on $X$
and $g_0$ a Ricci positive metric defined in an open neighbourhood of $Y$.  If the 1-jets of $g_0$ and $g_1$ are equal at every point in $Y$, then there exists a Ricci positive metric $\bar g$ on $X$ and numbers $\epsilon$, $\epsilon'$ with
$0<\epsilon'<\epsilon$, such that $\bar g|_{X\setminus N_\epsilon(Y)}$ agrees with
$g_1$ and $\bar g|_{N_{\epsilon'}(Y)}$ agrees with $g_0$. (Here, the $\epsilon$ and $\epsilon'$-neighbourhoods are defined with respect to the metric $g_1$.)
\end{theorem}
The case we are interested in is the situation where $Y$ is a point, so we are redefining the metric in an $\epsilon$-neighbourhood of the point. It is straightforward to check that the proof of Theorem \ref{David2002} continues to hold in this special case.

We remark that one might alternatively approach some of the results in this section, and also certain constructions in later sections, using ideas from \cite{Kor}. We note, for example, the analogy between Theorem 3.5 in that paper, and Corollary \ref{homotopy_equivalence} in the current paper.

We will need to use a few of the background details (in the special case $Y=\{x_0\} \subset M$) behind the proof of
Theorem \ref{David2002}. Let us assume that $g_1$ is a given Ricci positive metric on $M$, and that $g_0$ is a round metric of some radius defined in a neighbourhood of a basepoint $x_0 \in M$, and that the 1-jets of $g_0$ and
  $g_1$ agree at $x_0$. We introduce a cut-off function $f:\R \to \R$, which can be any smooth function satisfying:
\begin{equation*}
  f(r)=\left\{\begin{array}{ll}
  1 & r\le 1 \\
  0 & r\ge 2
\end{array}
\right.
\end{equation*}
with $0 \le f \le 1$ and $f'\le 0$. Suppose that the round metric $g_0$ is defined on $M$ in $B_{\epsilon}(x_0)$, a $g_1$-ball about $x_0$ of radius $\epsilon$. We define the following function on this ball:
\begin{equation*}
  \psi(x)=f(|x|^\lambda/\sqrt \lambda),
\end{equation*}
where $|x|$ denotes the $g_1$-distance of $x$ from $x_0$, and $\lambda \in (0,1)$ is a constant to be determined. We then define the following metric in $B_{\epsilon}(x_0)$:
\begin{equation*}
  \bar{g}(x)=\psi(x)g_0(x)+(1-\psi(x))g_1(x).
\end{equation*}
We note that $\bar{g}$ is smooth metric as a consequence of the fact that the functions $f$ and in turn $\psi$ are smooth. We also observe that 
\begin{equation*}
\bar{g} = \left\{\begin{array}{lll}
  g_1 & \mbox{if} & |x| \ge (2\sqrt \lambda)^{\frac{1}{\lambda}},\\
  g_0 & \mbox{if} & |x|<(\sqrt{\lambda})^{\frac{1}{\lambda}} \ .
\end{array}\right.
\end{equation*}
That this transition from $g_0$ to $g_1$ can take place within $B_{\epsilon}(x_0)$ follows from the fact that
\begin{equation}\label{lim}  
\lim_{\lambda\to 0^+} (2\sqrt \lambda)^{1/\lambda}=0,
\end{equation}
so given any $\epsilon>0$ we simply have to choose $\lambda$ sufficiently small so that $(2\sqrt\lambda)^{1/\lambda}<\epsilon.$

Of course we wish to alter the metric $g_1$ on $M$ by replacing it by $\bar{g}$ on $B_{\epsilon}(x_0)$.  However we need this change to preserve the positive Ricci curvature condition. The key result here
is \cite[Lemma 1.9]{Wr}, which adapted to our situation yields:
\begin{lemma} 
\label{Wr-lemma}{\rm \cite[Lemma 1.9]{Wr}}
  Let $u_1,u_2$ be $g_1$-unit tangent vectors at $x \in
  B_{\epsilon}(x_0)$.  Then setting $s=\psi(x)$ we
  have
\begin{equation*}
\Ric_{\bar{g}}(u_1,u_2)=s\Ric_{g_0}(u_1,u_2)+(1-s)\Ric_{g_1}(u_1,u_2)+A(u_1,u_2)
\end{equation*}
where
\begin{equation*}
  |A(u_1,u_2)| \le c\lambda^{1/2}|x|^\lambda
\end{equation*}
and $c$ depends on the dimension, the choice of $f$ and the metrics $g_0$ and $g_1$.
\end{lemma}
Thus we see that for $\lambda$ sufficiently small, the metric $\bar{g}$ has positive Ricci curvature, agrees with $g_1$ near the boundary of $B_{\epsilon}(x_0)$ (and hence fits smoothly into $(M\setminus B_{\epsilon}(x_0),g_1)$) and is round near $x_0$. We conclude that for the metric $g_1$ we can construct a metric in $\rd(M)$ by the above procedure. Furthermore, by performing a suitable global rescale if necessary, we can ensure the resulting metric actually lies in $\rdo(M).$ Notice that in the statement of Theorem \ref{David2002} we can take $\epsilon=(2\sqrt \lambda)^{1/\lambda}$ and $\epsilon'=(\sqrt \lambda)^{1/\lambda},$ for $\lambda$ sufficiently small depending on the metrics $g_0,g_1$. 

It follows easily from Lemma \ref{Wr-lemma} that we can widen the setting from positive Ricci curvature to $k$-positive Ricci curvature:  
\begin{proposition}
\label{kpos_rounding} 
Let $M^n$ be a manifold equipped with a $k$-positive Ricci metric $g_1,$ where $1 \le k \le n$. Then given any point $x_0 \in M,$ there are constants $R>0,$ $\epsilon$, $\epsilon'$ with $0<\epsilon'<\epsilon$, and a $k$-positive Ricci metric $\bar{g}$ on $M$ such that $\bar{g}$ is round of radius $R$ in the ball $B(x_0,\epsilon'),$ and agrees with $g_1$ in the complement of $B(x_0,\epsilon).$ Moreover, if $g_1$ is round of some radius $\rho$ in a neighbourhood of $x_0$, then we can take $R=\rho.$ More generally, we can arrange for the value of $R$ to depend smoothly on the metric, and $\epsilon,\epsilon'$ to depend continuously on the metric.
\end{proposition}

\begin{proof}
We begin by investigating the extent to which we can impose a round metric onto a neighbourhood of $x_0 \in M.$ We will do this by defining a local diffeomorphism between a neighbourhood of $x_0$ and a neighbourhood (say of the north pole) in a standard sphere, and then using this diffeomorphism to pull-back the round metric onto $M$.

We construct the diffeomorphism by introducing normal coordinate systems locally around $x_0$ in $M$ and around the north pole in the sphere. To do this in $M$, we consider the ordered basis $\{e_1,...,e_n\}$ for $T_{x_0}M$. Applying the Gram-Schmidt orthonormalization algorithm to this basis yields an ordered orthonormal basis $\{v_1,...,v_n\}$ in a way which depends smoothly on the metric $g_1$. This then generates a normal coordinate system in a neighbourhood of $x_0$.

To do the same thing on the sphere $S^n$, we begin by fixing once and for all an orthonormal basis $\{w_1,...,w_n\}$ for the tangent space at the north pole on the unit radius sphere. However we still have to fix the radius of the round metric we wish to use in our construction. Consider the $g_1$-unit vector $u:=e_1/|e_1|.$ We set the radius $R$ for the round sphere to be
\begin{equation*}
R=R(g_1)=\sqrt{\frac{n-1}{\Ric_{g_1}(u,u)}}.
\end{equation*}
It is clear that this depends smoothly on the metric, as claimed. This choice of $R$ also has the effect that should $g_1$ be round of radius $\rho$ near $x_0$, the metric $g_0$ will agree with this, i.e. we will have $R=\rho.$ Rescaling $\{w_1,...,w_n\}$ by a factor of $1/R$ then yields an orthonormal basis for the round metric of radius $R$. Finally, we use this basis to generate a normal coordinate system on the sphere.

The extent to which any normal coordinate system can be defined will be limited by the injectivity radius at the central point. In the case of a round sphere of radius $R$, the injectivity radius at any point is $\pi R$. For $g_1$ on $M$ we set
\begin{equation*}
\epsilon_0:=\min\{1, \pi R(g_1),\hbox{inj}_{g_1}(x_0)\}.
\end{equation*}
Notice that $\epsilon_0$ depends continuously on the metric.

Given this $\epsilon_0>0$, we can now construct a diffeomorphism between the $\epsilon_0$-balls about $x_0 \in M$ and about the north pole in $S^n$ simply by identifying points with the same normal coordinates. We then pull back the round metric of radius $R$ to $M$ along this diffeomorphism, and call the resulting locally defined metric $g_0.$ Notice that the 1-jets of $g_0$ and $g_1$ agree at $x_0$, by virtue of the nature of normal coordinate systems.

It is an elementary consequence of the Ricci curvature formula in Lemma \ref{Wr-lemma} that since $g_0$ is round and $g_1$ has $k$-positive Ricci curvature, then provided $\lambda$ is chosen small enough, the metric $\bar{g}$ will also have $k$-positive Ricci curvature. We therefore conclude that there exists $\lambda_0$ maximal in $(0,1]$ such that $\bar g$ has $k$-positive Ricci curvature for all $\lambda \in (0,\lambda_0).$ Set
\begin{equation*}
\lambda:=\min\Bigl\{\frac{\lambda_0}{4},\frac{1}{4}\epsilon_0^{2\lambda_0}\Bigr\},
\end{equation*}
so $\lambda$ varies continuously with the metric. We can then set
\begin{equation*}
\epsilon:=(2\sqrt \lambda)^{1/\lambda}, \text{ and } \epsilon':=(\sqrt \lambda)^{1/\lambda}.
\end{equation*}
Thus $\epsilon$ and $\epsilon'$ also depend continuously on the metric. Notice that by our choice of $\lambda$ we automatically have $\epsilon \le \epsilon_0.$ By the construction in \cite{Wr} outlined above, the metric $\bar{g}$ will then have all the desired properties: it will have $k$-positive Ricci curvature, be round of radius $R$ in an $\epsilon'$-ball around $x_0$, and will agree with the ambient metric $g_1$ towards the boundary of the $\epsilon$-ball.
\end{proof}

\begin{cor}
\label{cpt_family}
Given a compact family $K \subset \kRP(M),$ the association of a metric in $\rdk(M)$ as in Proposition \ref{kpos_rounding} can be achieved in such a way that the parameters $\epsilon,\epsilon'$ can be chosen uniformly for the entire family.
\end{cor} 

\begin{proof}
Begin by replacing the definition of $\epsilon_0$ in the proof of Proposition \ref{kpos_rounding} with the value
\begin{equation*}
\epsilon_0:=\min\{1,\inf_{g \in K} \pi R(g),\inf_{g \in K} \hbox{inj}(g)\},
\end{equation*}  Then define
\begin{equation*}
\lambda_0:=\inf_{g \in K} \lambda_0(g),
\end{equation*}
where $\lambda_0(g)$ is the value of $\lambda_0$ (as defined in the proof of Proposition \ref{kpos_rounding}) for the specific metric $g$. Note that by the compactness of $K$, all the infima above are positive.
Now define $\lambda$, and in turn $\epsilon,\epsilon'$ in terms of the above $\epsilon_0,$ $\lambda_0,$ as in the proof of the Proposition.
\end{proof}

We now consider the effect of moving the basepoint $x_0$ along a smooth path in $M$:
\begin{lemma}
\label{basepoint}
Given a $k$-positive Ricci metric $g$ on $M$ and a smooth path $\gamma(t)$ in $M$, $t \in [0,1]$, there are parameters $0<\epsilon'<\epsilon<1,$ a smooth function $R:[0,1]\to \mathbb{R}^+$, and a smooth path of metrics $g(t) \in \kRP(M)$, all dependent on $g$ and $\gamma$, with the following property. For each $t \in [0,1]$ the metric $g(t)$ is round of radius $R(t)$ in an $\epsilon'$-neighbourhood of the point $\gamma(t)$, and agrees with $g$ in the complement of the corresponding $\epsilon$-ball. (Here, the ball radius is measured with respect to $g$.)
\end{lemma}

\begin{proof} 
Given the $g$-unit vector $u:=\gamma'(0)/|\gamma'(0)|$ at $\gamma(0) \in M,$ consider the vector field $u(t)$ along $\gamma$ created by parallel translating $u.$ Using this in the rounding construction of Proposition \ref{kpos_rounding} we then obtain the function
\begin{equation*}
R(t)=\sqrt{\frac{n-1}{\Ric_{g}(u(t),u(t))}},
\end{equation*}
which clearly varies smoothly with $t$. The fact that the image of $\gamma$ is a compact subset of $M$ allows us to choose $\epsilon,\epsilon'$ uniformly for the whole path, as in Corollary \ref{cpt_family}. Then the rounding process about each $\gamma(t)$ with these parameters and $R(t)$ as above results in the desired smooth path of $k$-positive Ricci metrics.
\end{proof}  

Our next target is to show the inclusions: $ \rdok(M)\subset\rdk(M)\subset \kRP(M),$ are weak homotopy equivalences for any $1 \le k \le n.$ We do this by observing the respective relative homotopy groups all vanish. We begin with
 
\begin{proposition}
\label{prop:rounding}
Let $x_0 \in M$ represent a basepoint and let $K \subset \kRP(M)$ denote a compact family of metrics.  Then there is a homotopy of maps, $\eta_{s}: K\rightarrow \kRP(M)$, $s\in [0,1]$, which satisfies the following conditions:
\begin{enumerate}
\item[(i)] The map $\eta_{0}$ is the original inclusion map.
\item[(ii)] The image of the map $\eta=\eta_1$  lies in $\rdk(M)$.
\item[(iii)] For all $s\in [0,1]$, $\eta_{s}(K\cap \rdk(M))\subset \rdk(M)$.
\item[(iv)] For all $s\in [0,1]$, $\eta_{s}$ restricts as the identity map on $K\cap \rdok(M)$.
\end{enumerate}
\end{proposition}
\begin{proof}
We compose the inclusion map with the rounding map resulting from the construction in Proposition \ref{kpos_rounding}, after observing that by Corollary \ref{cpt_family} we can make uniform choices of $\epsilon,\epsilon'$ for all metrics in $K$. With these choices fixed, we obtain the desired smooth map:
\begin{equation*}
\eta: K \to
\rdk(M),
\end{equation*}
which immediately satisfies property (ii) above.

Before establishing the required homotopy $\eta_{s}, s\in [0,1]$, we will demonstrate that $\eta$ satisfies property (iv). Recall that if the metric $g_1$ in Proposition \ref{kpos_rounding} is round in a neighbourhood of $x_0$ with curvature $1/R^2$, then the same will be true for the rounded metric $\bar{g}.$  Moreover, if the starting metric is round throughout the whole $\epsilon$-ball, then the metric deformation procedure will have precisely no effect. For any metric in $K \cap \rdok(M)$, the $\epsilon$-ball in which the metric deformation takes place lies within the distance sphere about $x_0$ with intrinsic radius 1. (The intrinsic radius of an $\epsilon$-distance sphere in a sphere of radius $R$ is $R\sin(\epsilon/R),$ and it is easily checked that this is strictly less than one since $0<\epsilon \le 1.$) Thus, the map $\eta$ satisfies property (iv).

The homotopy $\eta_s$ for $s \in I$ is provided by the rounding construction in Proposition \ref{kpos_rounding} where the cut-off function $f(r)$ (defined after Theorem \ref{David2002}) is replaced by $s\cdot f(r).$ It is easy to see that the resulting metrics vary smoothly with $s$, and interpolate between the inclusion map and $\eta$. Since the choice of cut-off function influences all subsequent choices, it is possible that a given choice of $\lambda,$ $\epsilon$ and $\epsilon'$ which work in the case $s=1$ might not be suitable for other values of $s$. However, since the values of $s$ belong to a compact interval, it is clear that we can make uniform choices for $\lambda,\epsilon,\epsilon'$ which work for all metrics in $K$ and all $s \in [0,1].$ By the rounding comments in the paragraph above, we immediately see that conditions (iii) and (iv) are satisfied.
\end{proof}

\begin{theorem}
\label{weak-homotopy-rd}
If $M^n$ supports a metric of $k$-positive Ricci curvature for some $1 \le k \le n,$ then the inclusion maps $$\rdok(M)\hookrightarrow\rdk(M)\hookrightarrow \kRP(M)$$ are weak homotopy equivalences. In particular, if $M$ admits a Ricci positive metric then we have weak homotopy equivalences $\rdo(M) \hookrightarrow\rd(M)\hookrightarrow \RP(M).$
\end{theorem}

\begin{proof}

The statement of the Theorem is equivalent to the statement that the following homotopy groups vanish for all $i$:
\begin{equation*}
\pi_i(\kRP(M), \rdk(M))
 \text{ and } \pi_i(\rdk(M), \rdok(M)).
\end{equation*}
Let us consider the first family of homotopy groups. We set any metric $\tilde{g}\in \rdok(M)\subset \rdk(M)$ as basepoint. For any $i$, an element of this group takes the form of a continuous map
\begin{equation*}
  \alpha:D^i \to \kRP(M), \ \ \ \mbox{such that}
  \ \ \ \alpha|_{\partial D^i}:S^{i-1} \to \rdk(M).
\end{equation*}
We let some $p \in S^{i-1}$ lying the pre-image of $\tilde{g}$ act as  our basepoint for $D^i$.

The map $\alpha$ determines a compact family of $k$-positive Ricci metrics. Setting $K=\hbox{Im }\alpha$ in Proposition
\ref{prop:rounding}, we see that $\alpha$ is homotopic via the composition $\eta_s \circ \alpha$ to a map 
$$\bar{\alpha}:=\eta \circ \alpha:D^i \to \rdk(M).$$ Since each map in the homotopy fixes elements of $K\cap \rdok(M)$, this is a homotopy through based maps. Moreover, by property (iii) of Proposition
\ref{prop:rounding}, we know that at each stage in the homotopy, elements of $K\cap \rdk(M)$ are mapped into $\rdk(M)$. Thus, we have a null-homotopy and $\alpha$ represents the the zero class in $\pi_i(\kRP(M), \rdk(M))$.

We now consider the second set of homotopy groups. Setting $\tilde{g}$ as before we now let $\alpha$ denote
a continuous map
\begin{equation*}
  \alpha:D^i \to \rdk(M), \ \ \ \mbox{such that}
  \ \ \ \alpha|_{\partial D^i}:S^{i-1} \to \rdok(M).
\end{equation*}
Once again $p\in S^{i-1}$ in the pre-image of $\tilde{g}$ denotes the basepoint for $D^{i}$.
In order to establish that $\alpha$ is null-homotopic, we will apply
a progressive global scaling factor to the space $\rdk(M)$. This takes the form of a homotopy
\begin{equation*}
H(\beta):\rdk(M) \times [0,1] \to \rdk(M)
\end{equation*}
given by $H(\beta)(g,t)=(1+t\beta)^2g$ for some constant $\beta>0$. Clearly, for each $t$, $H(\beta)(-,t)$ maps $\rdok(M)$ to itself. We then precompose this with $\alpha$ to obtain a homotopy $H(\beta)\circ\alpha$. Since the image of ${\alpha}$ is compact, there is some value of $\beta$ for which the image of the composition $H(\beta)(-,1)\circ \alpha$ is a subset of $\rdok(M)$ as required. 
The problem with this approach, however, is that $H$ moves the basepoint metric. Nevertheless, $H$ provides a canonical path of metrics between the original basepoint $\tilde{g}$ and the shifted basepoint $(1+\beta)^2\tilde{g},$ namely $(1+t\beta)^2\tilde{g}$ for $t \in [0,1].$ Using this path we can replace $H$ by a based homotopy in the usual manner (see for example \cite[page 345]{Ha}). In this way we obtain a based homotopy between $\alpha$ and a map into $\rdok(M)$, showing that $\alpha$ is null-homotopic. As $\alpha$ is an arbitrary element in this homotopy group, we see that the homotopy group itself must vanish. The Proposition follows.
\end{proof}

Let $\bar{M}$ denote $M^n \setminus D^n$, and by $\bdy$ we will denote the space of $k$-positive Ricci metrics on $\bar{M}$ such that the boundary is concave, the boundary metric is round with (intrinsic) radius 1, and a neighbourhood of the boundary is round of radius greater than 1. Given any metric in $\rdok(M)$, removing the interior of the round disc about $x_0$ which is bounded by the sphere with intrinsic radius 1 results in an element of $\bdy$. The following result is then trivial, with
the maps being provided by removing, respectively gluing in round discs of the appropriate curvature:
\begin{lemma}
\label{homeo}
There is a homeomorphism $\rdok(M)
\cong \bdy.$  
\end{lemma}

\begin{cor}
\label{homotopy_equivalence}
If $M^n$ supports a metric of $k$-positive Ricci curvature for some $1 \le k \le n,$ then there are homotopy equivalences:
$$\kRP(M) \simeq \rdk(M)\simeq\rdok(M) \simeq \bdy.$$
 In particular, if $M$ admits a Ricci positive metric then the spaces $\RP(M)$, $ \rdo(M)$,  $\rdo(M)$ and ${\mathcal R}^{Ric>0}(\bar{M})^*$ are all homotopy equivalent.
\end{cor}

\begin{proof} 
By Lemma \ref{homeo} it suffices to show that $\kRP \simeq \rdk \simeq \rdok.$ Theorem \ref{weak-homotopy-rd} establishes this for  {\em weak} homotopy equivalence. In order to strengthen this to a genuine
homotopy equivalence, we observe that the space of all Riemannian metrics on a compact manifold is a metrizable space, and hence so are $\kRP(M)$ and $\rdok(M)$. For more details see, for example, \cite[1.38(c) and 1.46]{Ru}. Having
established metrizability, the existence of the desired homotopy equivalence now follows directly from \cite[Theorem 15]{Pa}.
\end{proof}


\section{$H$-space structures}\label{Hspacesection} 

In this section we will prove Theorem A. We will show that the space $\rdok(S^n)$ has an $H$-space structure when $n \ge 3$ and $k\ge 2$. The fact that this space is homotopy equivalent to $\kRP(S^n)$ means that by choosing any homotopy equivalence maps, we can pull back the $H$-space structure from $\rdok(S^n)$ to $\kRP(S^n)$, proving Theorem A. In order that $\rdok(S^n)$ be unambiguously defined, let us fix the basepoint $x_0$ of $S^n$ to be the north pole.

Recall from the Introduction that the surgery result of Wolfson implies that connected sums are possible within $Ric_k>0$ for any $2 \le k \le n.$ 
The $H$-space operation we will define below will be a special, modified version of the Wolfson construction. This is analogous to the approach taken by the first author in \cite{Wa}. The following technical lemma will be crucial to defining this special connected sum operation.

\begin{lemma}[Tube Lemma]\label{tube}
Given $n \ge 3$, $R>0$ and $\epsilon \in (0,R),$ set $\tau_{R,\epsilon}=R\cos^{-1}(\epsilon/R).$ Then there exist constants $\rho_{R,\epsilon} \in (0,R),$ $\ell_{R,\epsilon}>0$, $\kappa_{R,\epsilon}:=\frac{1}{2}+\frac{\pi}{4\cos^{-1}(\epsilon/R)}$, and a smooth warped product metric $g_{R,\epsilon}$ on $[\tau_{R,\epsilon},\ell_{R,\epsilon}] \times S^{n-1}$, all depending smoothly on $R$ and $\epsilon$, such that
\begin{enumerate}
\item[(i)] for $r \in [\tau_{R,\epsilon},\kappa_{R,\epsilon}\tau_{R,\epsilon}],$ $g_{R,\epsilon}=dr^2+R\cos^2(r/R)ds^2_{n-1}$;
\item[(ii)] for $r \in [\ell_{R,\epsilon}-1,\ell_{R,\epsilon}]$, $g_{R,\epsilon}=dr^2+\rho^2_{R,\epsilon}ds^2_{n-1}$;
\item[(iii)] $Ric_2(g_{R,\epsilon})>0$.
\end{enumerate}
\end{lemma}

Notice that the definition of $\kappa_{R,\epsilon}$ ensures that $\tau_{R,\epsilon}<\kappa_{R,\epsilon} \tau_{R,\epsilon} <R\pi/2.$  In particular this means that $\cos(\kappa_{R,\epsilon}\tau_{R,\epsilon}/R)>0.$
Let us denote the Riemannian manifold $([\tau_{R,\epsilon},\ell_{R,\epsilon}] \times S^{n-1},g_{R,\epsilon})$ by ${\mathcal T}_{R,\epsilon}$. A key point here is that for $r$ close to $\tau_{R,\epsilon}$, ${\mathcal T}_{R,\epsilon}$ looks like a radius $R$ round metric, and the $r=\tau_{R,\epsilon}$ boundary is a round sphere with intrinsic radius $\epsilon.$ In the current section, we will only need to consider $R>1$ and $\epsilon=1.$ For simplicity we will denote the resulting `tube' by ${\mathcal T}_R$, and drop the $\epsilon$ subscripts from the related quantities.

Before proceeding with the proof it is worth recalling a well known calculation for the Ricci curvature of a rotationally symmetric metric $dt^{2}+f(t)^{2}ds_{n-1}^{2},$ where $f:[0,L]\rightarrow (0,\infty)$ is a smooth function, (see page 69 of \cite{P}). Let $\partial_t, e_1,\cdots,e_{n-1}$ be an orthonormal frame with $\partial_t$ tangent to the interval $[0,L],$ and each $e_i$ tangent to the sphere $S^{n-1}.$ We have:
\begin{equation}\label{ricwarp}
\begin{split}
Ric(\partial_t)&=-(n-1)\frac{f''}{f},\\
Ric(e_i)&=(n-2)\frac{1-f'^2}{f^2}-\frac{f''}{f} \hspace{0.2cm}\text{, when $i=1,\cdots ,n-1$}.
\end{split}
\end{equation}
Thus, $$Ric_{2}=\min\left\{(n-2)\frac{1-f'^2}{f^2}-n\frac{f''}{f},\quad 2(n-2)\frac{1-f'^2}{f^2}-2\frac{f''}{f}\right\}.$$
We will be interested in the case when $n\geq 3$. Furthermore, we will insist that the function $f$ satisfy the condition that $ 0\leq |f'|\leq 1$. Thus the inequality $$ \frac{1-f'^2}{f^2} \ge 0$$
always holds, 
and where $f''<0$ it is clear that $Ric_2>0.$
If $f''\ge 0$, it is evident that the 
$2$-Ricci curvature takes the form:
\begin{equation*}\label{ric2warp}
Ric_{2}=(n-2)\frac{1-f'^2}{f^2}-n\frac{f''}{f}.
\end{equation*}
A simple calculation then shows that in order to obtain 2-positive Ricci curvature, it suffices to specify $f$ such that:
\begin{equation}\label{Ric2Ineq}
f''<\frac{1-f'^2}{\bar{n}f},
\end{equation}
where $\bar{n}=\frac{n}{n-2}\in (1, 3]$, since $n\geq 3$. 

Returning to Lemma \ref{tube}, we will need to make some further preparations before presenting the proof. This will consist of a collection of smaller technical results. The first of these leads to the existence of a $C^1$-tube which approximates the desired tube ${\mathcal T}_R$. This approximating tube then has to be smoothed: we do this in stages. Firstly we show how to smooth from $C^1$ to $C^2$, then from $C^2$ to $C^\infty$. Finally we combine these results to obtain a $C^1$ to $C^\infty$ smoothing which depends only on the parameter $R$.

The next result is a variant of \cite[Lemma 3.3]{EF}. The result in \cite{EF} deals with positive scalar curvature, and we have adapted the idea to apply to $Ric_k>0.$
\begin{lemma}\label{E_F}
Given $a>0,$ $b>0$ and $-1<c<0$, there exists a solution $h(t)$ to the initial value problem 
\begin{align*} 
h''&=\frac{1-h'^2}{ah}, \\
h(t_0&)=b, \\ 
h'(t_0&)=c, \\
\end{align*} 
for $t \in [t_0,T],$ some $T>t_0$ with $h(T)>0,$ $h'(T)=0.$
\end{lemma}

We remark that setting $a>3$ in the above lemma will give a function $h(t)$ which automatically satisfies the $Ric_2>0$ inequality (\ref{Ric2Ineq}). 

\begin{proof}
By the classical Picard-Lindel\"of Theorem, there exists a solution to the initial value problem at least in some interval $t \in (t_0-\epsilon,t_0+\epsilon).$ 
Observe that if the solution $h(t)$ is defined for some $t_1>t_0,$ then $h(t_1)>0$, as we must have $h(t)>0$ for all $t \in [t_0,t_1]$ in order for the ODE to be defined.

Set $$C(t)=\frac{h^{-1/a}(t)}{\sqrt{1-h'^2}(t)}.$$ By differentiating $C(t)$, we see that $C(t)$ is constant if $h(t)$ satisfies the above ODE. Moreover, we claim that throughout its domain of definition, $$h(t) \ge C(t_0)^{-a}.$$ To see this, we observe that for any $t$ at which $h(t)$ is defined, $$\frac{h^{-1/a}(t)}{\sqrt{1-h'^2(t)}}=C(t)=C(t_0)$$ with the second equality following from the fact that $C$ is a constant function. We then see that $h'(t) \in (-1,1)$ for all $t$ for which the solution is defined, and rearranging we obtain $$h^{-1/a}(t)=C(t_0)\sqrt{1-h'^2(t_1)} \le C(t_0).$$ The claim now follows immediately.

From the above analysis we see that the only way that the solution can fail to exist for all $t>t_0$ is if there exists $t_1 \in (t_0,\infty)$ such that $$\lim_{t \to t_1^{-}} h(t) =\infty.$$ (We can rule $-\infty$ as a limit since $h(t)$ is bounded below.) In this case, by the intermediate value theorem, since $h'(t_0)<0$, there exists $T\in (t_0,t_1)$ with $h'(T)=0.$

On the other hand, if the solution exists for all $t>t_0$ and we always have $h'(t)<0$, then it follows that $$h''(t) \ge \frac{1-h'^2(t_0)}{ah(t_0)}.$$ (To see this, note that since $h' \in (-1,1)$ we have $h''>0$, and thus $h'^2$ is decreasing. By assumption $h$ is also decreasing, hence the inequality.)
Integrating we see that $$h'(t) \ge \frac{1-c^2}{ab}(t-t_0)+c,$$ and so $h'(t)$ hits zero at some finite time after $t_0$. This contradicts the assumption that $h'(t)<0.$ We conclude that $h'$ must change sign, establishing the existence of a $T>t_0$ with the desired properties.
\end{proof}


\begin{lemma}[$C^1$ to $C^2$ smoothing]
\label{C1_to_C2}
Consider a function $g(t)$ which is smooth for all $t$ except at the point $t=t_0$ where it is $C^1$. Suppose further that $g''(t)$ remains bounded as $t \to t_0^\pm.$ Given $\nu>0$, there exists $\delta_0\in (0,1]$ depending continuously on $\nu$, $g$ and its first and second derivatives near $t=t_0$, such that for all $\delta\in (0,\delta_0),$ there is a function $\tilde{g}(t)$ which agrees with $g(t)$ outside a $\delta$-neighbourhood of $t=t_0$, and has the following properties.
\begin{enumerate}
\item It is piecewise smooth with precisely two non-smooth points at $t=t_0\pm \delta,$ at which it is $C^2$.
\item It has a $C^2$-continuous dependence on $\delta$.
\item It is $C^1$ $\nu$-close to $g$.
\item For $t \in (t_0-\delta,t_0+\delta)$, $\tilde{g}''$ interpolates between $g''(t_0-\delta)$ and $g''(t_0+\delta).$
\end{enumerate}
\end{lemma}

\begin{proof}
Consider initially any small $\delta>0.$ Over the interval $(t_0-\delta,t_0+\delta)$ we define $\tilde{g}$ to be a quintic polynomial. The coefficients of this polynomial are completely determined by the values taken by $g$ and its first and second derivatives at $t=t_0-\delta$ and $t=t_0+\delta,$ so as to create a $C^2$-function. The precise formula for this polynomial is displayed explicitly in \cite[proof of Theorem 2]{BWW}. Moreover, it is also shown there that as $\delta \to 0,$ $\tilde{g}$ converges to $g$ in the $C^1$-norm, and for $\delta$ sufficiently small, the second derivative of $\tilde{g}$ interpolates between its values at $t_0\pm \delta.$ One then sees from the analysis in \cite{BWW} that what it means to be `sufficiently small' in this context depends continuously on $\nu$ and on the $C^2$-behaviour of $g$ either side of $t=t_0$.
\end{proof}

At several points in the sequel we will need to consider a smooth `step' function. Let us fix such a function once and for all: let $\phi:[0,1] \to [0,1]$ be any choice of smooth function such that for some small $\epsilon>0$, $\phi(t)=0$ for $t \in [0,\epsilon],$ $\phi(t)=1$ for $t \in [1-\epsilon,1]$  and $0 \le \phi' <2$ throughout.


\begin{lemma}[$C^2$ to $C^\infty$ smoothing]
\label{C2_to_smooth}
Suppose that $f(t)$ and $g(t)$ are real-valued, real analytic functions in a neighbourhood of $t=t_0$. If $$h(t)=
\begin{cases}
&f(t) \text{ if } t \le t_0 \\
&g(t) \text{ if } t\ge t_0
\end{cases}
$$
is a $C^2$-function, then given any $\epsilon>0$ there is a smooth function $\tilde{h}$, depending smoothly on $\epsilon$, which agrees with $h$ outside an $\epsilon$-neighbourhood of $t=t_0$, and for which $|\tilde{h}-h|_{C^2} \to 0$ as $\epsilon \to 0.$
\end{lemma}

\begin{proof}
Let $\psi(t):[-1,1] \to [0,1]$ be defined by $\psi(t):=\phi((t+1)/2).$ Then set $\psi_\epsilon(t):=\psi(t/\epsilon)$. Thus $\psi_\epsilon$ is a smooth step function defined for $t \in [-\epsilon,\epsilon].$ We set $$\tilde{h}=
\begin{cases}
& \bigl(1-\psi_\epsilon(t-t_0)\bigr)f(t)+\psi_\epsilon(t-t_0)g(t) \text{ for } t \in [t_0-\epsilon,t_0+\epsilon] \\
& h(t) \text{ otherwise.}
\end{cases}.
$$
Clearly $\tilde{h}$ is smooth. We must consider $\tilde{h}-h.$ To this end, note that if $M_1$ and $M_2$ are such that $|\psi'| \le M_1$ and $|\psi''|\le M_2,$ then we have $|\psi_\epsilon'| \le \epsilon^{-1}M_1$ and $|\psi_\epsilon''| \le \epsilon^{-2}M_2.$ Also, since $h$ is $C^2$ at $t=t_0$ we have $f(t_0)=g(t_0),$ $f'(t_0)=g'(t_0)$ and $f''(t_0)=g''(t_0).$ It then follows from the real analyticity assumption that $g(t)-f(t)=O(t^3).$ Now $$\tilde{h}(t)-h(t)=
\begin{cases}
& \psi_{\epsilon}(t)(g(t)-f(t)) \text{ for } t \le t_0 \\
&(1-\psi_{\epsilon}(t))(f(t)-g(t)) \text{ for } t\ge t_0.
\end{cases}
$$
Thus on an $\epsilon$-neighbourhood of $t_0$ we have 
\begin{align*}
&|\tilde{h}-h|_{C^0}=O(\epsilon^3); \\
&|\tilde{h}'-h'|_{C^0}=O(\epsilon^{-1}.\epsilon^3+\epsilon^2)=O(\epsilon^2); \\
&|\tilde{h}''-h''|_{C^0}=O(\epsilon^{-2}.\epsilon^3+\epsilon^{-1}.\epsilon^2+\epsilon)=O(\epsilon).\\
\end{align*}
Therefore $|\tilde{h}-h|_{C^2} \to 0$ as $\epsilon \to 0.$ 
\end{proof}


\begin{cor}[Main smoothing]
\label{smoothing}
Let $f$ be a function which is real analytic except possibly at $t=t_0$, where it is (at least) $C^1$. If $f$ is precisely $C^1$ at $t=t_0,$ assume that $f''$ is bounded as $t \to t_0^\pm$. Suppose further that for $t>t_0$ and for $t<t_0$, $f$ satisfies the $Ric_2>0$ inequality (\ref{Ric2Ineq}). Then there exists $\alpha_0\in (0,1]$, depending continuously on $f$ and its first and second derivatives near $t_0$, such that for all $\alpha \in (0,\alpha_0)$, $f$ can be smoothed over an interval $(t_0-\alpha,t_0+\alpha)$ to give a function $\tilde{f}$ which satisfies the $Ric_2>0$ inequality. Moreover, if $f$ is actually smooth at $t=t_0,$ $|\tilde{f}-f|_{C^2} \to 0$ as $\alpha \to 0$.
\end{cor}

\begin{proof} 
First we consider smoothing $f$ from $C^1$ to $C^2$ using Lemma \ref{C1_to_C2}. Denote the result of this smoothing by $\bar{f}$. For $\nu>0$, Lemma \ref{C1_to_C2} guarantees us a $\delta_0>0$ such that for any $\delta \in (0,\delta_0)$ we have $|\bar{f}-f|_{C^1}<\nu$, together with the property that over the interval $(t_0-\delta,t_0+\delta)$, $\bar{f}''$ interpolates between $f''(t_0\pm \delta).$ Thus if $\nu$ is chosen sufficiently small, it follows that $\bar{f}$ will satisfy the $Ric_2>0$ inequality. Denote by $\nu_0>0$ the supremum within the set $(0,1]$ of those $\nu$ with this property. Now set $\delta_0$ to be the constant produced by Lemma \ref{C1_to_C2} corresponding to $\nu=\nu_0/2$. As both the $C^2$ smoothing and the $Ric_2>0$ inequality have a $C^2$ dependence on the original function near $t=t_0$, we see that $\nu_0$, and therefore $\delta_0$, vary continuously as $f$ varies in a $C^2$-continuous fashion.

Next, we smooth $\bar{f}$ from a $C^2$ to a $C^\infty$-function $\tilde{f}$ using Lemma \ref{C2_to_smooth}. Clearly, since $f$ is real analytic, by construction $\bar{f}$ is also real analytic, so Lemma \ref{C2_to_smooth} applies. Given that $\bar{f}$ satisfies the $Ric_2>0$ inequality, there exists $\epsilon_1$ maximal in $(0,1]$ such that if a smooth function $\theta$ satisfies $|\tilde{f}-\theta|_{C^2}<\epsilon_1$ on the interval $[t_0-1,t_0+1]$, then $\theta$ must also satisfy the $Ric_2>0$ inequality. 

Set the value of $\alpha_0$ in the statement of Corollary \ref{smoothing} to be $\delta_0$. For any $\alpha \in (0,\alpha_0),$ set $\delta=\alpha/2$, and $\epsilon=\min\{\epsilon_1,\alpha/4\}.$

In the $C^1$-to-$C^2$ smoothing, our choice of $\delta_0$ and $\delta$ guarantee that the resulting $C^2$ function $\bar{f}$ satisfies the $Ric_2>0$ inequality, and is smooth away from $t=t_0\pm \alpha/2.$ The $C^2$-to-$C^\infty$ smoothing then smooths each of these non-smooth points over an interval of length $\epsilon,$ hence the whole deformation takes places over the interval $(t_0-\alpha/2-\epsilon,t_0+\alpha/2+\epsilon).$ As $\epsilon\le \alpha/4,$ this interval is contained in $(t_0-\alpha,t_0+\alpha)$ as required. As $\epsilon \le \epsilon_1,$ the function $\tilde{f}$ also satisfies the $Ric_2>0$ inequality.



The final task is to show that if $f$ is actually smooth at $t=t_0,$ $|\tilde{f}-f|_{C^2} \to 0$ as $\alpha \to 0$. It follows from Lemma \ref{C1_to_C2} in this case that as $\delta \to 0$, $|\bar{f}-f|_{C^2} \to 0.$ From Lemma \ref{C2_to_smooth} it follows that $|\tilde{f}-\bar{f}|_{C^2} \to 0$ as $\epsilon \to 0.$ Given that $\alpha$ controls the size of both $\delta$ and $\epsilon$ (as indicated above), the claim now follows immediately from the triangle inequality.
\end{proof} 
 
\begin{proof}[Proof of Lemma \ref{tube}]
Let $$\zeta_{R,\epsilon}:=\frac{\frac{3\pi}{4}+\frac{1}{2}\cos^{-1}(\epsilon/R)}{\frac{\pi}{2}+\cos^{-1}(\epsilon/R)}.$$ It is easily checked that $\kappa_{R,\epsilon}\tau_{R,\epsilon}<\kappa_{R,\epsilon}\tau_{R,\epsilon}\zeta_{R,\epsilon}<R\pi/2.$

We begin by constructing the main part of the tube, which will run from $r=\kappa_{R,\epsilon} \tau_{R,\epsilon}\zeta_{R,\epsilon}$ to $r=\ell_{R,\epsilon}-\frac{3}{2}.$  To do this we use Lemma \ref{E_F} with $t_0=\kappa_{R,\epsilon} \tau_{R,\epsilon}\zeta_{R,\epsilon},$ $b=R\cos(\kappa_{R,\epsilon} \tau_{R,\epsilon} \zeta_{R,\epsilon}/R),$ $c=-\sin(\kappa_{R,\epsilon} \tau_{R,\epsilon}\zeta_{R,\epsilon}/R),$ and $a=4,$ to produce a function $h(r)$. Given the constant $T$ produced by Lemma \ref{E_F} in this case, we set $\ell_{R,\epsilon}=T+\frac{3}{2},$ and $\rho_{R,\epsilon}=h(T).$ The corresponding smooth warped product metric $dr^2+h^2(r)ds^2_{n-1}$ will have $Ric_2>0$, and give a $C^1$ join at $r=\kappa_{R,\epsilon}\tau_{R,\epsilon} \zeta_{R,\epsilon}$ and $r=\ell_{R,\epsilon}-\frac{3}{2}$ with the $Ric_2>0$ warped product metrics specified in points (i) and (ii) of Lemma \ref{E_F}.

Our next task is to smooth the metric at these two non-smooth points. To do this we use Corollary \ref{smoothing}. First note that by the Cauchy-Kovalevskaya Theorem, the function $h$ is real-analytic, hence the Corollary applies here. Notice also that our $C^1$-warped product scaling function depends smoothly on $R$ and $\epsilon$, in the sense that the two $C^1$-points vary smoothly with these parameters, and at any point in the interior of a $C^\infty$ region, the output values vary smoothly as the function varies with $R, \epsilon$. Thus for each of the two non-smooth points we obtain constants $\alpha_1(R,\epsilon),$ $\alpha_2(R,\epsilon)$ corresponding to $\alpha_0$ in the smoothing Corollary, which depend continuously on $R,\epsilon$. Set $$\alpha_0(R,\epsilon)=\min\Big\{\alpha_1(R,\epsilon),\alpha_2(R,\epsilon),\frac{\kappa_{R,\epsilon}\tau_{R,\epsilon}(\zeta_{R,\epsilon}-1)}{100},\frac{\ell_{R,\epsilon}-\frac{3}{2}-\kappa_{R,\epsilon}\tau_{R,\epsilon} \zeta_{R,\epsilon}}{100},\frac{1}{100}\Big\},$$ so $\alpha_0(R,\epsilon)$ also varies continuously with $R$ and $\epsilon$. Note that the last three entries in the above minimum expression are included to keep the smoothing localized around the non-smooth points. In particular, the metric will still take the form $dr^2+\rho^2_R ds^2_{n-1}$ for $r\in [\ell_{R,\epsilon}-1,\ell_{R,\epsilon}]$, and $dr^2+R^2\cos^2(r/R)$ for $r \in [\tau_{R,\epsilon},\kappa_{R,\epsilon}\tau_{R,\epsilon}]$.

Now choose any smooth function $\alpha(R,\epsilon)$ with $0<\alpha(R,\epsilon)<\alpha_0(R,\epsilon)$ for all $R>0$ and $\epsilon \in (0,R).$ We will choose to smooth each $C^1$-point of our scaling function according to the smoothing Corollary with $\alpha=\alpha(R,\epsilon)$ for each of the non-smooth points. This preserves the $Ric_2>0$ condition, and the resulting tube $\mathcal{T}_{R,\epsilon}$ and its associated parameters all depend smoothly on $R$ and $\epsilon$.
\end{proof}

We will need one further result before we can discuss $H$-space structures on $\rdok(S^n)$. This, together with the tubes $\mathcal{T}_R={\mathcal T}_{R,1}$ (with $R>1$), will enable us to define a notion of multiplication. First, a preliminary lemma.

\begin{lemma}[Stretching lemma]
\label{stretch}
Given a warped product metric $dt^2+f^2(t)ds^2_{n-1}$ on $[a,b] \times S^{n-1},$ there exists $N>0$ (depending on $f$) such that the metric $dt^2+f^2(t/N)ds^2_{n-1}$ on $[aN,bN] \times S^{n-1}$ has $Ric_2>0.$
\end{lemma}

\begin{proof}
For any $t_0 \in [a,b]$, we can clearly arrange for the metric $dt^2+f^2(t/N)ds^2_{n-1}$ to be $C^2$-arbitrarily close in a neighbourhood of $t=t_0N$ to the metric $dt^2+f^2(t_0)ds^2_{n-1}$ by choosing $N$ sufficiently large. As the latter metric has $Ric_2>0,$ the result follows.
\end{proof}

The following result is now immediate:

\begin{cor}
Given $c_1,c_2>0$, there exists $\lambda_0=\lambda_0(c_1,c_2)$ minimal in $[1, \infty)$ such that for all $\lambda>\lambda_0$, the metric 
\begin{equation}\label{connect}
dt^2+\Bigl(\bigl((1-\phi(t/\lambda)\bigr)c_1+\phi(t/\lambda)c_2\Bigr)^2ds^2_{n-1}
\end{equation}
on $[0,\lambda] \times S^{n-1}$ has $Ric_2>0.$
\end{cor}

\begin{definition}
Fix a smooth function $\lambda_\infty:{\mathbb R}^+ \times {\mathbb R}^+ \to (1,\infty)$ with the property that $\lambda_\infty(c_1,c_2)>\lambda_0(c_1,c_2)$ for all $c_1,c_2>0$. Given any $c_1,c_2>0,$ the `connecting piece' $C(c_1,c_2)$ is the product manifold $[0,\lambda_\infty] \times S^{n-1}$ equipped with the metric (\ref{connect}) with $\lambda=\lambda_\infty$.
\end{definition}

\begin{remark}
Notice that $C(c_1,c_2)$ varies smoothly with $c_1,c_2.$ In the special case $c_1=c_2=c$ we have $C(c,c)=([0,1] \times S^{n-1}, dt^2+c^2ds^2_{n-1}).$ Below we will need to consider the connecting piece in the situation where $c_1=\rho_R=\rho_{R,1}$ and $c_2=\rho_{100}=\rho_{100,1},$ where $\rho_{R,1}$ and $\rho_{100,1}$ are as in Lemma \ref{tube}.
\end{remark}

In order to define the $H$-space multiplication on $\rdok(S^n)$, we will not glue two metrics directly via a connected sum, but via an intermediate space, or `docking station'. The concept of a docking station originated in work of the first author (\cite{Wa}; see also \cite{BB}), based on a suggestion by Boris Botvinnik. In our case the docking station will be modelled on a round sphere of radius 100, $S^n(100),$ which has three distinguished points. The first of these is the north pole $x_0,$ which is the basepoint. The other two points, $z_1$ and $z_2,$ will be diametrically opposed points on the equator. Consider the open balls about $z_1,z_2$ of radius $100\sin^{-1}(1/100).$ The boundary of these balls is a sphere with intrinsic radius 1. Remove these balls from $S^n(100)$ to leave a manifold with two boundary components. 
Notice that this space is still based at the north pole $x_0.$ This will be our docking station.

\begin{definition}
\label{H_mult}
For any $n \ge 3$ and $2 \le k \le n,$ we define a map $$\sigma:\rdok(S^n) \times \rdok(S^n) \to \rdok(S^n)$$ as follows. Consider metrics $h_1,h_2 \in \rdok(S^n).$  Suppose that in a neighbourhood of the basepoint the metrics $h_1$ and $h_2$ are round with radius $R_1$ respectively $R_2.$ For each metric we remove an open ball about the basepoint so as to leave a round concave boundary with intrinsic radius 1. Attach tubes ${\mathcal T}_{R_1},$ ${\mathcal T}_{R_2}$ to these punctured spheres in the obvious way, and then attach the connecting pieces $C(\rho_{R_1},\rho_{100})$ and $C(\rho_{R_2},\rho_{100})$ to the free ends of the respective tubes. 
Next, glue a copy of ${\mathcal T}_{100}$ to the free end of each connecting piece, and finally, for $i=1,2,$ glue the resulting arrangement for $h_i$ to the boundary sphere of the docking station corresponding to $z_i$. Denote the closed Riemannian manifold thus produced by $\hat{\sigma}(h_1,h_2).$ To aid the reader, the Riemannian manifold $\hat{\sigma}(h_1,h_2)$ is depicted in Fig. \ref{Multiplication}. 
Now $h_1$ and $h_2$ are metrics defined on a `standard' copy of the sphere $S^n$, and in order to have a well-defined product $\sigma(h_1,h_2)$, we must identify $\hat{\sigma}(h_1,h_2)$ with a metric on the standard $S^n$. Moreover we must do this in a prescribed way which varies smoothly with $h_1,h_2$. Below we describe a connected sum contraction procedure which provides a means of doing this. The final metric we obtain will be our `product' $\sigma(h_1,h_2).$
\end{definition}

\begin{figure}[!htbp]
\vspace{2cm}
\hspace{1.5cm}
\begin{picture}(0,0)
\includegraphics{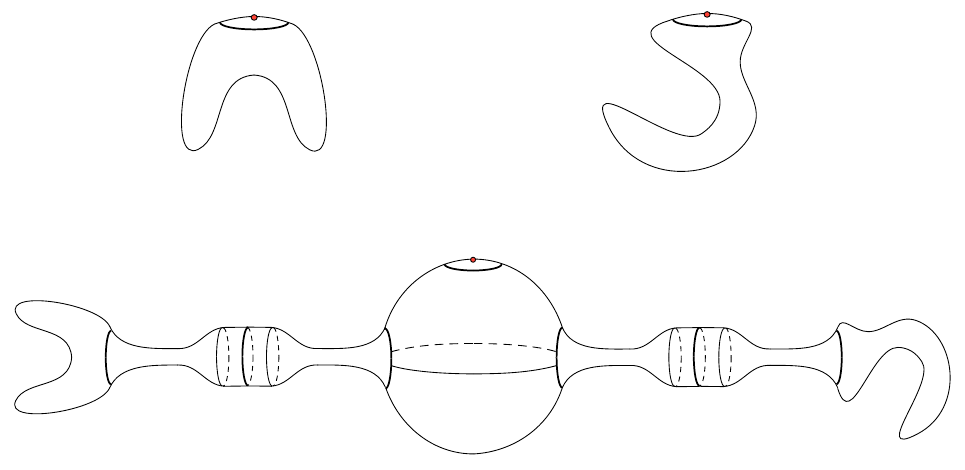}%
\end{picture}
\setlength{\unitlength}{3947sp}
\begingroup\makeatletter\ifx\SetFigFont\undefined%
\gdef\SetFigFont#1#2#3#4#5{%
  \reset@font\fontsize{#1}{#2pt}%
  \fontfamily{#3}\fontseries{#4}\fontshape{#5}%
  \selectfont}%
\fi\endgroup%
\begin{picture}(5079,1559)(1902,-7227)
\put(2600,-5000){\makebox(0,0)[lb]{\smash{{\SetFigFont{10}{8}{\rmdefault}{\mddefault}{\updefault}{\color[rgb]{0,0,0}$h_1$}%
}}}}
\put(5500,-5000){\makebox(0,0)[lb]{\smash{{\SetFigFont{10}{8}{\rmdefault}{\mddefault}{\updefault}{\color[rgb]{0,0,0}$h_2$}%
}}}}
\put(3750,-6100){\makebox(0,0)[lb]{\smash{{\SetFigFont{10}{8}{\rmdefault}{\mddefault}{\updefault}{\color[rgb]{0,0,0}$\hat{\sigma}(h_1,h_2)$}%
}}}}

\end{picture}%
\caption{The metrics $h_1,h_2 \in \rdok(S^n)$ (top), and the `product' manifold $\hat{\sigma}(h_1,h_2)$ (bottom)}
\label{Multiplication}
\end{figure}
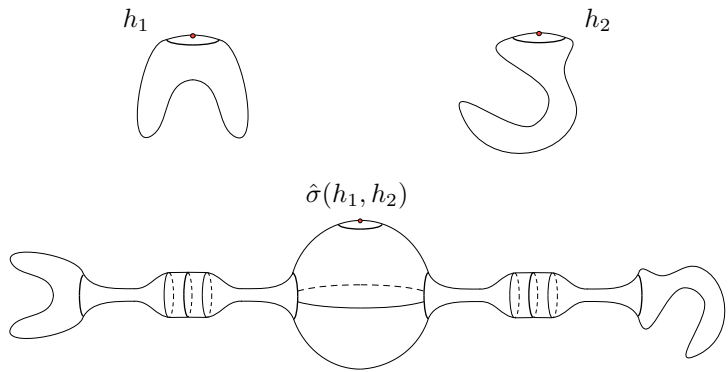

\noindent{\bf Connected sum contraction.} Given metrics $h_1,h_2 \in \rdok(S^n),$ assume that in a neighbourhood of the basepoints these metrics restrict to round metrics of radii $R_1,R_2$ respectively, and we denote the round discs about the basepoints for which the boundary is a sphere of intrinsic radius $\epsilon _i\in (0,R_i)$ by $D_{1,\epsilon_1},D_{2,\epsilon_2}.$ Form the Riemannian manifold 
\begin{equation}\label{connectsum}
(S^n\setminus D_{1,\epsilon_1},h_1) \cup {\mathcal T}_{R_1,\epsilon_1} \cup C(\rho_{R_{1,\epsilon_1}},\rho_{R_{2,\epsilon_2}}) \cup {\mathcal T}_{R_2,\epsilon_2}\cup (S^n\setminus D_{2,\epsilon_2},h_2).
\end{equation} 
Here, and in what follows, we will also use the symbols $h_i$ to denote metrics restricted to discs within $S^n$.

Our aim is to introduce parametrizations into the discs $$(S^n\setminus D_{1,\epsilon_1},h_1) \cup {\mathcal T}_{R_1,\epsilon_1} \cup C(\rho_{R_1,\epsilon_1},\rho_{R_2,\epsilon_2}) \cup {\mathcal T}_{R_2,\epsilon_2},$$ and $D_{2,\epsilon_2}$, and then use these parametrizations to define a diffeomorphism from the latter disc to former, which we use to pull back the metric. We can then replace the metric $h_2|_{D_{2,\epsilon_2}}$ in the second sphere with the pull-back metric, and provided the pull-back and original metrics agree near the boundary of this disc, the resulting metric will be smooth. In this way we can identify the metric on (\ref{connectsum}) with a metric on the second sphere in a natural way. We will also ensure that the diffeomorphism we use depends smoothly on the parameters $R_i$ and $\epsilon_i.$ In what follows, we will label the spheres $S_1$ and $S_2$ to remove any ambiguity.

We begin by temporarily equipping the sphere $S_1$ with the round metric of radius $R_1$. This allows us to introduce a distance parameter $r$ from the point diametrically opposite the basepoint. Let us fix this parametrization, and impose it now on the Riemannian manifold $(S_1,h_1).$ Of course, this parametrization agrees (up to a shift) with the parametrization used in a neighbourhood of the basepoint in previous constructions. In particular this parametrization naturally extends throughout $(S_1\setminus D_{1,\epsilon_1},h_1) \cup {\mathcal T}_{R_1,\epsilon_1} \cup C(\rho_{R_1,\epsilon_1},\rho_{R_2,\epsilon_2}) \cup {\mathcal T}_{R_2,\epsilon_2}\cup (S_2\setminus D_{2,\epsilon_2},R_2^2ds^2_n)$ and, moreover, agrees up to a shift with the parametrization on $S_2$ (or rather on $S_2\setminus D_{2,\epsilon_2}$) which results from measuring distance from the basepoint of $S_2$ with respect to $R_2^2ds^2_n.$

With respect to the distance parameter from the basepoint on $(S_2,R_2^2ds^2_n)$, let $\delta_2(R_2,\epsilon_2)$ denote the parameter value at the boundary of $D_2$. Similarly, let $\delta_1$ denote the value of $r$ at the boundary of $(S_1\setminus D_{1,\epsilon_1},h_1) \cup {\mathcal T}_{R_1,\epsilon_1} \cup C(\rho_{R_1,\epsilon_1},\rho_{R_2,\epsilon_2}) \cup {\mathcal T}_{R_2,\epsilon_2}.$ Finally, let $d_{R_1,\epsilon_1}$ be the value of $r$ at the boundary of $(S_1\setminus D_{1,\epsilon_1},h_1).$ 

The diffeomorphism we wish to define will be rotationally symmetric (with respect to standard metrics), and will thus be determined by a choice of diffeomorphism $\Delta=\Delta(R_1,\epsilon_1,R_2,\epsilon_2): [0,\delta_2] \to [0,\delta_1].$ In order for the pull-back metric to agree with the original near the boundary we will need $\Delta$ to satisfy $\Delta'(t)=1$ for $t$ close to $\delta_2,$ and in order for it to be smooth at the centre point we will also need $\Delta$ to be odd at $t=0.$ We will also arrange that $(S_1\setminus D_{1,\epsilon_1},h_1)$ undergoes compression by a uniform factor during this process. 

\begin{lemma}\label{contraction}
For any $\delta_1,\delta_2>0$ and $a \in (0,\min\{\delta_1,\delta_2\})$, there is an orientation preserving diffeomorphism $\Delta_a=\Delta_a(\delta_1,\delta_2)$ with $\Delta_a:[0,\delta_2]\to [0,\delta_1],$ $\Delta_a'(t)=1$ for $t \in [\delta_2-a/2,\delta_2],$ and $\Delta_a'(t)$ constant for $t \in [0,\delta_2-a]$ (ensuring that $\Delta_a$ is odd at $t=0$). Moreover, $\Delta_a$ depends smoothly on the parameters $\delta_1,\delta_2$ and $a$.
\end{lemma}

\begin{proof}
We begin by constructing a function $\Phi_a(t):[0,\delta_2] \to [0,\delta_1]$ depending smoothly on $\delta_1,\delta_2,a,$ with $\Phi_a(0)=\delta_1,$ $\Phi_a(\delta_2)=0,$ $\Phi'_a<0,$ $\Phi'_a=-1$ for $t \in [0,a/2]$, and $\Phi'_a=c,$ some $c \in (-\infty,0]$, for $t \in [a,\delta_2]$. Having established the existence of $\Phi_a,$ we set $\Delta_a(t):=\Phi_a(\delta_2-t).$ It is easily checked that $\Delta_a$ then satisfies the required properties.


Recall the bump function $\phi:[0,1]\to [0,1]$ defined before Lemma \ref{C2_to_smooth}. Extend this in the obvious way (i.e. with constant value 1 for $t \ge 1$ and 0 for $t \le 0$) to a function $\bar{\phi}:{\mathbb R} \to {\mathbb R}.$ For $c \in (-\infty,0]$ and $t \in {\mathbb R}^+$ we set
\begin{equation}
\Theta_a(t):=\delta_1+\int_0^t (-1)\Bigl[1-\bar{\phi}\Bigl(\frac{u-a/2}{a/2}\Bigr)\Bigl]+c\bar{\phi}\Bigl(\frac{u-a/2}{a/2}\Bigr) \,\,du.
\end{equation}
It is clear that $\Theta_a$ depends smoothly on $\delta_1,$ $a$ and $c$. Moreover we have $\Theta_a(0)=\delta_1,$ $\Theta'_a(t)=-1$ for $t \in [0,a/2],$ $\Theta'_a(t)=c$ for $t \ge a,$ and $\Theta'_a<0$ for all $t$ when $c<0.$ Now for $c=0$ we clearly have $\Theta_a(\delta_2)>0$, and for $c$ sufficiently large and negative we will have $\Theta_a(\delta_2)<0.$ Therefore by the intermediate value theorem there exists a (unique) value of $c$ for which $\Theta_a(\delta_2)=0.$ Set $\Phi_a$ to be the restriction of $\Theta_a$ for this particular value of $c$ with domain $[0,\delta_2]$ and target $[0,\delta_1].$
\end{proof}

For our purposes $\delta_1,\delta_2$ will depend smoothly on $R_i,\epsilon_i$, as indicated previously. For technical reasons later on, we will need $\Delta_a'(t)$ to be constant for $t \in [0,\Delta_a^{-1}(d_{R_{1,\epsilon_1}})].$ As $\Delta_a'$ is constant by construction for $t \in [0,\delta_2-a],$ we then need to arrange for $a < \delta_2-\Delta_a^{-1}(d_{R_{1,\epsilon_1}}).$ The point of this is to ensure that the map $\Delta$ we will construct compresses $(S_1\setminus D_{1,\epsilon_1},h_1)$ by a uniform factor. (We will need to undo this effect at some point, and the uniformity makes the process straightforward.) 

\begin{cor}
There exists $a_0=a_0(R_1,\epsilon_1,R_2,\epsilon_2)>0$ depending continuously on $R_i,\epsilon_i$, such that for all $a \in (0,a_0)$ we have  $a < \delta_2-\Delta_a^{-1}(d_{R_{1,\epsilon_1}}).$
\end{cor}

\begin{proof}
It is easily checked that as $a \to 0,$ the diffeomorphism $\Delta_a$ converges pointwise to the diffeomorphism $\Delta_0(t):=t\delta_1/\delta_2.$ Since $\Delta_0^{-1}(d_{R_{1,\epsilon_1}})=d_{R_{1,\epsilon_1}}\delta_2/\delta_1$, by choosing $a$ sufficiently small we can ensure that $\delta_2-\Delta_a^{-1}(d_{R_{1,\epsilon_1}})$ is arbitrarily close to $\delta_2(1-d_{R_{1,\epsilon_1}}/\delta_1)>0.$ Thus the inequality $a < \delta_2-\Delta_a^{-1}(d_{R_{1,\epsilon_1}})$ will hold for all $a$ sufficiently small. Set $a_0$ to be the least upper bound of the set of $a\in (0,\min\{\delta_1,\delta_2\})$ such that this inequality holds. It is clear that $a_0$ depends at least continuously on $R_1,$ $\epsilon_1,$ $\delta_1$ and $\delta_2.$ But the $\delta_i$ depend smoothly on $R_i$ and $\epsilon_i$, hence $a_0=a_0(R_1,\epsilon_1,R_2,\epsilon_2)$ as claimed.
\end{proof} 

Let us choose and fix any smooth function $a(R_1,\epsilon_1,R_2,\epsilon_2)$ such that $$0<a(R_1,\epsilon_1,,R_2,\epsilon_2)<a_0(R_1,\epsilon_1,R_2,\epsilon_2)$$ for all $R_i>0,$ $\epsilon_i \in (0,R_i).$
\begin{definition}\label{Delta}
Let $\Delta=\Delta(R_1,\epsilon_1,R_2,\epsilon_2)$ be given by $\Delta=\Delta_{a(R_1,\epsilon_1,R_2,\epsilon_2)}.$ Denote by $\bar{\Delta}$ the rotationally symmetric diffeomorphism determined by $\Delta$.
\end{definition}

Thus $\bar{\Delta}$ maps $D_{2,\epsilon_2} \subset S_2$ onto the disc comprising the connected sum arrangement involving $S_1$. Pulling back the metric via $\bar{\Delta}$ then, in effect, contracts the connected sum onto $D_{2,\epsilon_2}.$ We will refer to $\bar{\Delta}$ as the {\it contracting map}, (even though, technically speaking, it is $\bar{\Delta}^{-1}$ that is actually a contracting map).

Note that for the remainder of this section, unless stated otherwise, it should be assumed that a disc $D^n$ removed from a sphere $S^n$ is a disc centred on the basepoint within a constant curvature region for which the boundary sphere has intrinsic radius 1. 

\begin{remark}
In order to complete Definition \ref{H_mult} we have to specify the metric $\sigma(h_1,h_2)$. We obtain $\sigma(h_1,h_2)$ from $\hat{\sigma}(h_1,h_2)$ by applying the connected sum contraction procedure twice, once for each of the metrics $h_1$ and $h_2$. In each case, the sphere $S_2$ in the connected sum contraction should be taken to be the docking station sphere $(S^n,100^2ds^2_n)$. The disc $D_2$ will then be centred on the point $z_1$ for attaching $h_1,$ and on $z_2$ for connecting $h_2$. Thus the points $z_1,z_2$ should temporarily be considered the basepoints of the respective discs $D_2$ for the purposes of the above construction. Also, notice that the locations of the connected sum operations in Definition \ref{H_mult} are far-removed from the actual basepoint $x_0$ (in the docking station), and this ensures that the resulting metric still belongs to $\rdok(S^n).$
\end{remark}

\begin{remark}\label{id_ending} 
When we prove the existence of a homotopy identity element below, two further observations about the connected sum contraction procedure will be relevant. Firstly, in the special case of Lemma \ref{contraction} where $\delta_1=\delta_2$, for all $a  \in (0,\delta_1)$ the relevant value of $c$ is $-1$, and the resulting diffeomorphism $\Delta:[0,\delta_1] \to [0,\delta_1]$ is simply the identity map. Secondly, we can apply the contraction idea to contract any rotationally symmetric disc onto another, provided the given metrics agree near the boundaries of the respective discs. Thus given a smoothly varying family of such discs described by a parameter $s$, with metrics fixed in a neighbourhood of the boundary, and with the disc radius given by a function $\delta_1(s),$ we can use the corresponding family of diffeomorphisms $\Delta(s)$ to pull back to $D_2$.
\end{remark}

\begin{proposition}
\label{H_commute}
The binary operation $\sigma$ is homotopy commutative.
\end{proposition}

\begin{proof}
Consider arbitrary metrics $h_1,h_1 \in \rdok(S^n).$ We need to show that $\sigma(h_1,h_2) \simeq \sigma(h_2,h_1).$
To achieve this, simply rotate $S^n$ about its north-south axis so as to swap the positions of $z_1$ and $z_2$. This clearly carries the metric $\sigma(h_1,h_2)$ smoothly to the metric $\sigma(h_2,h_1)$ as required.
\end{proof}


\begin{proposition}
\label{identity}
The round metric $100^2ds^2_n$ is a homotopy identity element for the binary operation $\sigma.$
\end{proposition}

The proof of this proposition comes down to showing that for an arbitrary $h\in \rdok(S^n)$ we can construct a $\Ric_{2}>0$ isotopy which moves this metric to the metric $\hat{\sigma}(h,g_{100})$ in a way which fixes all of $h$ outside the standard round region; see Fig. \ref{HomotIdFirst} for a depiction of these metrics. Such an isotopy is easily shown to give rise to the desired homotopy equivalence. Before continuing with a proof of this Proposition, we first make a construction which will be crucial.

\begin{figure}[!htbp]
\vspace{0cm}
\hspace{0.5cm}
\begin{picture}(0,0)
\includegraphics{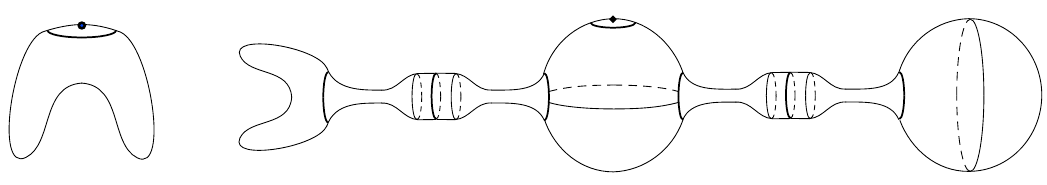}%
\end{picture}
\setlength{\unitlength}{3947sp}
\begingroup\makeatletter\ifx\SetFigFont\undefined%
\gdef\SetFigFont#1#2#3#4#5{%
  \reset@font\fontsize{#1}{#2pt}%
  \fontfamily{#3}\fontseries{#4}\fontshape{#5}%
  \selectfont}%
\fi\endgroup%
\begin{picture}(5079,1559)(1902,-7227)
\put(4300,-6200){\makebox(0,0)[lb]{\smash{{\SetFigFont{10}{8}{\rmdefault}{\mddefault}{\updefault}{\color[rgb]{0,0,0}$\sigma(h,g_{100})$}%
}}}}
\put(2200,-6200){\makebox(0,0)[lb]{\smash{{\SetFigFont{10}{8}{\rmdefault}{\mddefault}{\updefault}{\color[rgb]{0,0,0}$h$}%
}}}}
\end{picture}%
\caption{The arbitrary metric $h\in\rdok(S^n)$ and the metric $\hat{\sigma}(h,g_{100})$}
\label{HomotIdFirst}
\end{figure}
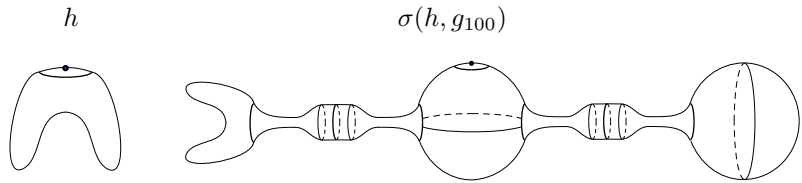

\noindent{\bf Warped product deformation procedure.} The aim is to provide a method for smoothly deforming a disc equipped with a warped product metric $dt^2+\mu^2(t)ds^2_{n-1}$ into a round disc, preserving the $Ric_2>0$ condition. To this end, we consider a smooth, positive function $\mu(t),$ $t \in [t_0,L]$ for some $0 \le t_0 <L$, which satisfies the $Ric_2>0$ inequality (\ref{Ric2Ineq}) for all $t$ (with the derivatives understood in a one-sided sense at the end of the domain interval). We will also assume that $\mu'\in (-1,1)$. Here we view $t=t_0$ as corresponding to the {\it boundary} of the disc, and $t=L$ as the centre. In order for the proposed deformation to end with a round disc, it is neccessary for $dt^2+\mu^2(t)ds^2_{n-1}$ to be round near the disc boundary.

We will construct a smooth one-parameter family of smooth functions $\gamma_s(t)$ based on the function $\mu$. Here $s \in [\tau_0,L_0]$ for some $t_0 <\tau_0<L_0 \le L,$ and the domain of $\gamma_s(t)$ is $[t_0,L(s)]$ where $L(s)\in [s,\infty)$ varies smoothly with $s$. The idea is that the resulting warped products $([t_0,L(s)] \times S^{n-1}, dt^2+\gamma^2_s(t)ds^2_{n-1})$ interpolate smoothly between the original disc and a round disc. Moreover, we will arrange for this one-parameter family of metrics to be constant with respect to $s$ near the disc boundary.

We proceed as follows. We first construct a $C^1$-approximation $\bar{\gamma}_s$ to the desired function $\gamma_s$. We then apply earlier smoothing results to smooth $\bar{\gamma}_s$ to $\gamma_s$ over an interval, and in a manner, which depends only on $s$ and the original function $\mu(t).$

We begin by setting $$\bar{\gamma}_s(t)=
\begin{cases} 
\mu(t) & \text{ if } t \le s \\
\frac{\mu(s)}{\sqrt{1-\mu'^2(s)}}\cos\Bigl(\frac{\sqrt{1-\mu'^2(s)}}{\mu(s)}(t-s)+\sin^{-1}(-\mu'(s))\Bigr) & \text{ if } t\ge s.
\end{cases}
$$

It is easily checked that $\gamma_s$ is a piecewise smooth function, which is $C^1$ at $t=s,$ its only non-smooth point. We will set the upper domain parameter for $\gamma_s$, $L(s)$, to be equal to the smallest value of $t\ge s$ for which the above cosine expression has a zero. This clearly varies smoothly with $s$.

It is automatic that for $t\in [t_0,s)$ and for $t\in(s,L(s)]$, $\bar{\gamma}_s$ satisfies the $Ric_2>0$ inequality (\ref{Ric2Ineq}): in the former case this is an assumption on $\mu$, and in the later this is immediate since the second derivative of the cosine function is negative.

By Corollary \ref{smoothing}, for each $s$ there exists a number $\alpha_0(s),$ maximal in the interval $(0,\min\{(\tau_0-t_0)/2,(L(s)-s)/2\}],$ such that for any $\alpha \in (0,\alpha_0(s))$, the function $\bar{\gamma}_s$ can be smoothed over the interval $t \in (s-\alpha,s+\alpha)$ in such a way that the resulting function still satisfies the $Ric_2>0$ inequality (\ref{Ric2Ineq}). Moreover $\alpha_0(s)$ varies continuously with $s$. 
(Note that the value of $\alpha_0$ is chosen so that $t_0<s-\alpha<s+\alpha<L(s)$ for all $s \in [\tau_0,L_0].$) 
As $[\tau_0,L_0]$ is compact, $\min_{s \in [\tau_0,L_0]} \alpha_0(s)$ exists. If we set $\alpha=\frac{1}{2}\min \alpha_0(s)$, then the
smoothing can be performed over the interval $t \in (s-\alpha,s+\alpha)$ for every $s \in [\tau_0,L_0].$ Let $\gamma_s(t)$ be the family of functions resulting from this smoothing over intervals of length $2\alpha.$ This family is clearly smoothly dependent on $s$.

In the proof of Proposition \ref{identity}, and again in Section \ref{LSS}, we will apply this construction to the situation where the function $\mu$ is a warped product scaling function which depends smoothly on the parameters $R,\epsilon$, as will $t_0$ and $\tau_0$. It follows easily that the corresponding number $\alpha$ will vary continuously with respect to $R,\epsilon$. Let us denote this $\alpha_{R,\epsilon}$. We can therefore choose a smooth function $\alpha(R,\epsilon)$ with $0<\alpha(R,\epsilon)<\alpha_{R,\epsilon}$ for $R>0$, $\epsilon \in (0,R)$, and using this function to dictate the smoothing, we obtain a family of smooth functions $\gamma_s=\gamma_{s,R,\epsilon}$ which satisfy the inequality (\ref{Ric2Ineq}) and vary smoothly with respect to $s,$ $R$ and $\epsilon$. 
\begin{obs}\label{observation}
For any $s \in [\tau_0,L_0],$ the corresponding warped product metric $dt^2+\gamma^2_s(t)ds^2_{n-1}$ is round for $t \in [s+\alpha(R,\epsilon),L(s)].$ 
Consequently, this round region contains a distance sphere about the end-point with intrinsic radius $\epsilon':=\gamma_s(s+\alpha(R,\epsilon)),$  and the dependence of $\epsilon'$ on $R,\epsilon,s$ is smooth. 
\end{obs}
Using the above construction we can prove the following.

\begin{lemma}\label{blemish}
For given $R>0$ and $\epsilon \in (0,R)$, let $\tau_{R,\epsilon}$ and $\kappa_{R,\epsilon}$ be as in Lemma \ref{tube}. Consider a function $\mu(t),$ $t \in [\tau_{R,\epsilon},L]$, which satisfies the inequality (\ref{Ric2Ineq}) as in the construction above. If for $t \in [\tau_{R,\epsilon},\kappa_{R,\epsilon} \tau_{R,\epsilon}]$ we have $\mu(t)=R\cos(t/R)$, then there exists a one-parameter family of smooth curves $\gamma_s(t)$ for $s \in [\tau_{R,\epsilon},L_0]$, some given $L_0 \in (\kappa_{R,\epsilon} \tau_{R,\epsilon},L],$ which satisfy (\ref{Ric2Ineq}), depend smoothly on $s$, $R$ and $\epsilon$, and such that for $s$ suitably close to $\tau_{R,\epsilon}$ we have $\gamma_s(t)=R\cos(t/R)$, (so $\gamma_s$ is independent of $s$). 
\end{lemma}

\begin{proof}
This largely follows from the above construction by setting $t_0=\tau_{R,\epsilon}$ and $\tau_0=\kappa_{R,\epsilon} \tau_{R,\epsilon},$ which yields a family of functions $\gamma_s$ wth $s \ge \kappa_{R,\epsilon}\tau_{R,\epsilon}.$ The task is therefore to extend this family to $s \in [\tau_{R,\epsilon},\kappa_{R,\epsilon} \tau_{R,\epsilon}]$ in such a way that close to the left-hand end of this interval, $\gamma_s(t)=R\cos(t/R).$ To aid the reader we provide a depiction in Fig. \ref{WarpDef} below.

\begin{figure}[!htbp]
\vspace{1cm}
\hspace{2cm}
\begin{picture}(0,0)
\includegraphics{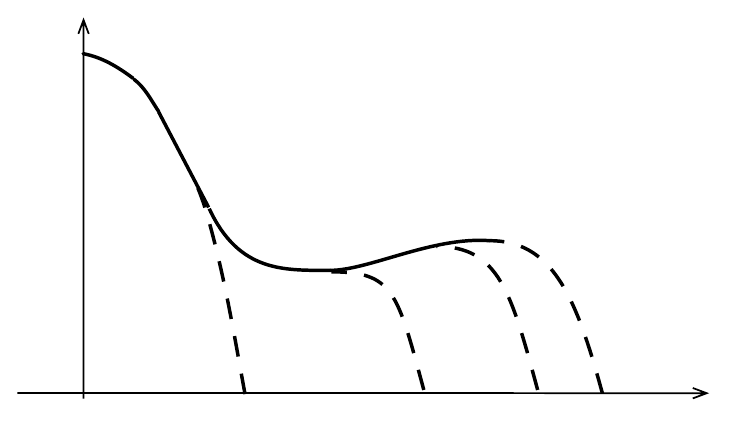}%
\end{picture}
\setlength{\unitlength}{3947sp}
\begingroup\makeatletter\ifx\SetFigFont\undefined%
\gdef\SetFigFont#1#2#3#4#5{%
  \reset@font\fontsize{#1}{#2pt}%
  \fontfamily{#3}\fontseries{#4}\fontshape{#5}%
  \selectfont}%
\fi\endgroup%
\begin{picture}(5079,1559)(1902,-7227)
\put(3700,-6500){\makebox(0,0)[lb]{\smash{{\SetFigFont{10}{8}{\rmdefault}{\mddefault}{\updefault}{\color[rgb]{0,0,0}$\gamma_s$}%
}}}}
\put(2600,-5500){\makebox(0,0)[lb]{\smash{{\SetFigFont{10}{8}{\rmdefault}{\mddefault}{\updefault}{\color[rgb]{0,0,0}$\mu$}%
}}}}
\put(5150,-7200){\makebox(0,0)[lb]{\smash{{\SetFigFont{10}{8}{\rmdefault}{\mddefault}{\updefault}{\color[rgb]{0,0,0}$t$}%
}}}}
\end{picture}%
\caption{The effect of the warped product deformation procedure (dashed curves)}
\label{WarpDef}
\end{figure}
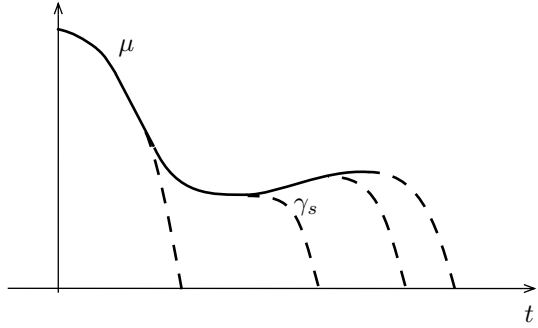

We can certainly define $\gamma(s)$ in exactly the same way when $s \in [\tau_{R,\epsilon},\kappa_{R,\epsilon} \tau_{R,\epsilon}]$, provided we are prepared to ignore the fact that the left-hand end of the $\bar{\gamma}_s$-to-$\gamma_s$ smoothing interval will lie below $t=\tau_{R,\epsilon}$ for $s$ close to $\tau_{R,\epsilon}.$ Notice that the function $\bar{\gamma}_s(t)$ is equal to $R\cos(t/R)$ for $t \in [\tau_{R,\epsilon},R\pi/2].$ However, for consistency we have to apply our smoothing procedure (as in Corollary \ref{smoothing}) to this already smooth curve. This modifies the curve in a non-trivial way, and thus for $s$ close to $\tau_{R,\epsilon}$ our task is to systematically undo this modification.

We approach this task mindful of the last claim in Corollary \ref{smoothing}: that if $f$ is smooth, as $\alpha \to 0,$ $|\tilde{f}-f|_{C^2} \to 0.$ Thus if $\alpha$ is chosen sufficiently small,  the smooth homotopy $(1-\phi(x))\tilde{f}(t) +\phi(x)f$ between $\tilde{f}$ and $f$ satisfies the inequality (\ref{Ric2Ineq}) for all $ x \in [0,1].$

Turning our attention back to the curves $\gamma_s(t)$, given the fixed form of $\mu(t)$ for $t \in [\tau_{R,\epsilon},\kappa_{R,\epsilon} \tau_{R,\epsilon}]$, we can assume without loss of generality that the function $\alpha(R,\epsilon)$ (as in the above construction) has been chosen small enough so that the $s$-indexed homotopy $$(1-\theta(s))R\cos(t/R)+\theta(s)\gamma_s(t)$$ where $$\theta(s):=\phi\Bigl( \frac{s-\tau_{R,\epsilon}(\kappa_{R,\epsilon}+1)/2}{\tau_{R,\epsilon}(\kappa_{R,\epsilon}-1)/2}\Bigr),$$ for $s \in [\tau_{R,\epsilon}(\kappa_{R,\epsilon}+1)/2,\kappa_{R,\epsilon} \tau_{R,\epsilon}]$ and $t \in [\tau_{R,\epsilon},\pi R/2]$, satisfies (\ref{Ric2Ineq}) for all $s$ and $t$. By construction, this homotopy smoothly `unwrinkles' $\gamma_s$ to the desired cosine function as $s$ approaches $\tau_{R,\epsilon}.$ 
Moreover it is easily checked that $\tau_{R,\epsilon}(\kappa_{R,\epsilon}+1)/2-\alpha(R,\epsilon)>\tau_{R,\epsilon}$, and hence the smoothing interval for $\bar{\gamma}_s$ is always contained in $(\tau_{R,\epsilon},L(s))$ for all $s \in [\tau_{R,\epsilon}(\kappa_{R,\epsilon}+1)/2,L_0].$
The claim about the smooth dependence of $\gamma_s$ on $R$ and $\epsilon$ is clear.
\end{proof}

\begin{remark}\label{start}
In the proof of Proposition \ref{identity} below, and again in Section \ref{LSS}, we will use the warped product deformation procedure on a function $\mu(t)$ which takes the form $\mu(t)=100\cos\bigl((t-L+50\pi)/100\bigr)$ for $t \in [L-50\pi-1,L].$ (Note that $L>50\pi+1.$) In this situation we will set $L_0=L-50\pi.$ Exactly the same argument as employed in the proof of Lemma \ref{blemish} shows that there exists a constant $c$, which here is independent of all other parameters, such that provided $\alpha(R,\epsilon)<c$ for $R>0,\epsilon \in (0,R),$ (which without loss of generality we will assume), then there is an analogous homotopy between $\gamma_s$ and $\mu$ for $s \in [L-50\pi-1,L-50\pi]$ say, satisfying (\ref{Ric2Ineq}) at each stage. In other words, we can - and will - assume that $\gamma_s$ agrees with $\mu$ (and is therefore independent of $s$) for $s$ suitably close to $L-50\pi$. 
\end{remark}

\begin{proof}[Proof of Proposition \ref{identity}]
As this argument is quite long, for the convenience of the reader we will break it into a number of steps. The general strategy throughout is to give a smooth deformation of Riemannian manifolds, starting with $\hat{\sigma}(h,e)$, and ending with $(S^n,h)$. This will give a path in $\rdok(S^n)$ provided that at each stage in the process we have a smoothly varying diffeomorphism from $S^n$ with which to pull back the metric. Note that throughout this argument we will assume $R>1$ and $\epsilon=1,$ so $\epsilon$ will be omitted from the notation.

Let us assume that $h$ is round with radius $R$ in a neighbourhood of the basepoint.

The following terminology will be useful. There is an isometric inclusion of $(S^n\setminus D^n,h)$ into $\hat{\sigma}(h,e).$ We will refer to the image of this inclusion as the `$h$-part' of  $\hat{\sigma}(h,e).$ Likewise, there is an isometric inclusion of $$(S^n\setminus D^n,h)\cup {\mathcal T}_R \cup C(\rho_R,\rho_{100}) \cup {\mathcal T}_{100}$$ into $\hat{\sigma}(h,e),$ and we will refer to the image of this as the `extended $h$-part'. Similarly for the (extended) $e$-part of $\hat{\sigma}(h,e).$ Given a diffeomorphism $S^n \to \hat{\sigma}(h,e),$ we will use the same terms for the corresponding parts of $S^n$ equipped the pull-back metric, as no confusion will arise. We will also use this terminology in analogous situations where the meaning is clear.
\smallskip

\noindent{\it Step 1.}
We begin by considering $\hat{\sigma}(h,e).$ Introduce a parameter $t$ locally into the manifold, so that half the docking station and the extended $e$-part of the manifold can be described as having a warped product metric  $dt^2+\mu^2(t)ds^2_{n-1}$ for some function $\mu(t)$, with $t \in [0,L]$ for some $L$. We will assume that the basepoint in the docking station has $t$-parameter 0, and that the `$e$-end' of the manifold corresponds to $t=L.$

Denote the connected-sum contraction diffeomorphisms (as in Definition \ref{Delta}) used to form $\sigma(h,e)$ from $\hat{\sigma}(h,e)$ by $\bar{\Delta}_h$ and $\bar{\Delta}_e$ (so $\bar{\Delta}_h=\bar{\Delta}(R,100)$ and $\bar{\Delta}_e=\bar{\Delta}(100,100)$). We can extend this pair of diffeomorphisms trivially over the body of the docking station in $\hat{\sigma}(h,e)$ to obtain a global diffeomorphism $S^n \to \hat{\sigma}(h,e)$, with $\sigma(h,e)$ being the pull-back of the metric on $\hat{\sigma}(h,e)$ via this map.

Next, use Lemma \ref{blemish} and Remark \ref{start} to smoothly deform the extended $e$-part of $\hat{\sigma}(h,e)$ back to the docking station. For this we need to set $R=100$ and $\epsilon=1$, so that the $t$-value at the boundary of the docking station in this arrangement is $\tau_{100}$. We also need to set $L_0=L-50\pi.$

At each stage in this contraction (parametrized by $s \in [\tau_{100},L_0]$) we have a quantity $\delta_1(s)$ as in Lemma \ref{contraction}, which measures the distance along the warped product manifold from the docking station boundary sphere to the endpoint of the $e$-disc, and this varies smoothly with $s$. For each value of $s$ there is a corresponding diffeomorphism $\bar{\Delta}_s$ (as in Definition \ref{Delta}). Pulling the metric back via this map, and contracting the extended $h$-part of $\hat{\sigma}(h,e)$ back by $\bar{\Delta}_h$, then produces a smooth path of metrics in $\rdok(S^n)$ starting with $\sigma(h,e).$ This path ends with a metric we will denote $\sigma(h,\cdot)$. By Remark \ref{id_ending} this metric agrees with the round metric of radius 100 outside the image of  $\bar{\Delta}^{-1}_h$, and is thus the pull-back of the natural metric on the manifold $$\hat{\sigma}(h,\cdot):=(S^n\setminus D^n,h)\cup {\mathcal T}_R \cup C(\rho_R,\rho_{100}) \cup{\mathcal T}_{100}\cup (S^n\setminus D^n,100^2ds^2_n)$$ via the trivial extension of $\bar{\Delta}_h$ over the docking station sphere.

We have now, in effect, smoothly eliminated the extended $e$-part of the metric from both $\hat{\sigma}(h,e)$ and $\sigma(h,e)$. Notice that in doing so, we have not moved the basepoint, nor made any metric change near the basepoint. In particular, we still have a distance sphere about the basepoint with an intrinsic round metric of radius one. 

Our task is now to smoothly modify the metric $\sigma(h,\cdot)$ back to the original $(S^n,h).$ In order to do this, we need pay special attention to the basepoint $x_0$, and the next step is carried out for precisely this reason.
\smallskip

\noindent{\it Step 2.} Bearing in mind that on $(S^n,100^2ds^2_n)$, a distance sphere with intrinsic radius 1 is the boundary of an embedded  disc of radius $100\sin^{-1}(1/100)$, we choose and fix a path from the point $z_1$ (about which the extended $h$-part is attached to the docking station - see Definition \ref{H_mult}) to the south pole (i.e. antipodal to the basepoint). We now smoothly deform the manifold $\hat{\sigma}(h,\cdot)$ by sliding the point of attachment of the extended $h$-part along the path, from its original position (determined by $z_1$) to being centred over the south pole. We do this by removing a disc of radius $100\sin^{-1}(1/100)$ about each point of the path in turn, and gluing in the extended $h$-part. We will assume that the path has been chosen so that every point is at a distance greater than $200\sin^{-1}(1/100)$ from the basepoint. This ensures that the deformation does not interfere with the basepoint or its neighbourhood. Let us denote the final Riemannian manifold by $\hat{\sigma}(h,\cdot)_S.$ At each point in the deformation, there is a corresponding diffeomorphism from the standard sphere, given by trivially extending the contraction map for the extednded $h$-part located at the appropriate point. Pulling back the metric by this corresponding family of diffeomorphisms then gives a smooth path in $\rdok(S^n)$, starting with $\sigma(h,\cdot)$ and ending with the extended $h$-part of the metric over the south pole. Let us denote this final diffeomorphism by $\bar{\Delta}_S:S^n \to \hat{\sigma}(h,\cdot)_S$, and the resulting pull-back metric on $S^n$ by $\sigma(h,\cdot)_S.$
\smallskip

We now need to smoothly deform $\sigma(h,\cdot)_S$ to $(S^n,h)$. We will do this in two stages. Firstly we will smoothly adjust the diffeomorphism $\bar{\Delta}_S,$ so that at the end of this deformation, the pull-back metric restricted to the complement of some disc $D_{x_0}$ is precisely $(S^n\setminus D^n,h).$ As a consequence of this `stretching' of the $h$-part of metric, the pull-back of the metric on ${\mathcal T}_R \cup C(\rho_R,\rho_{100})\cup {\mathcal T}_{100}\cup (S^n \setminus D^n,100^2ds^2_n)$ becomes squashed into $D_{x_0}.$ Secondly, we smoothly adjust the metric on $D_{x_0}$ to give the round metric of radius $R$. Our approach here is essentially the same as for deforming the extended $e$-part of the original arrangement back to the docking station: we will work with $\hat{\sigma}(h,\cdot)_S$ and deform this back to $S^n$, pulling-back the metric to $S^n$ throughout. However this time we will be reducing the docking station sphere back to the $h$-sphere. This requires us to take special care of the basepoint and its neighbourhood. For the moment however we will ignore this issue.
\smallskip

\noindent{\it Step 3.}
As in the connected-sum contraction procedure, we will introduce a global parameter $t$ ({\it ignoring the previous use of this symbol earlier in the proof}) into the manifold $\hat{\sigma}(h,\cdot)_S$. Given that $h$ is round with radius $R$ in a neighbourhood of its basepoint, we temporarily replace the $h$-part of $\hat{\sigma}(h,\cdot)_S$ by $(S^n\setminus D^n,R^2ds^2_{n-1}).$ With respect to this background metric, we introduce $t$ to measure distance from the south pole of this sphere. Suppose that $t \in [0,L']$. `Replacing' $R^2ds^2_{n-1}$ by the metric $h$, we continue to use the same parametrization throughout.

Similarly, we introduce a parameter $u$ into the standard sphere $S^n$: we let $u$ measure distance from the {\it south pole} in $(S^n,100^2ds^2_n),$ so that the basepoint corresponds to $u=100\pi.$

The diffeomorphism $\bar{\Delta}_S$ is determined by a diffeomorphism $\Delta_S:[0,100\pi] \to [0,L'].$ The value of $t \in [0,L']$ corresponding to the boundary where the $h$-part of $\hat{\sigma}(h,\cdot)_S$ meets the tube ${\mathcal T}_R$ is easily calculated to be $t_R:=R(\pi-\sin^{-1}(1/R)).$ 
Under the diffeomorphism $\bar{\Delta}^{-1}_S,$ we see that $(S^n\setminus D^n,h)$ is uniformly compressed into the disc centred on the south pole of the standard sphere corresponding to $u \in [0,\Delta^{-1}_S(t_R)].$ Our aim is to `uncompress' this by smoothly stretching the metric around the appropriate part of the standard sphere via a one-parameter family of diffeomorphisms $\bar{\theta}_x:S^n \to S^n$, for $x \in [0,1].$

Clearly, $t=t_R$ on $(S^n,R^2ds^2_n)$ corresponds to $u=100t_R/R$ when the underlying spaces are identified. (So $D_{x_0}$ will correspond to $u \in [100t_R/R,100\pi]$.) We therefore need to stretch the $u$-interval $[0,\Delta^{-1}_S(t_R)]$ within $[0,100\pi]$ by a progressive uniform factor, so that in the end, the fully-stretched interval coincides with $[0,100t_R/R].$ We must then extend in some way so that we have a smooth one-parameter family of diffeomorphisms $\theta_x:[0,100\pi] \to [0,100\pi]$, which then determines the diffeomorphisms $\bar{\theta}_x:S^n \to S^n$ in the obvious way.

We define $\theta_x,$ $x \in [0,1],$ as follows. Set $\theta_0(u)=u$ for all $u \in [0,100\pi].$ We will define $\theta_1$ so as to linearly map $[0,\Delta^{-1}_S(t_R)]$ onto $[0,100t_R/R]$. For $u$ close to $100\pi$ we will also demand that $\theta_1$ behaves linearly, and of course must satisfy $\theta_1(100\pi)=100\pi.$ We must therefore fashion a transition between these two linear parts, in a manner which depends smoothly on $R$. In order for the transition to be complete before $u=100\pi$, let us choose the transition interval to be $\frac{1}{100^2}(100\pi-100t_R/R)=\frac{1}{100}\sin^{-1}(1/R).$ Proceeding in a similar manner to the proof of Lemma \ref{contraction}, we set $$\theta_1(u)=
\int_0^u \Bigl[1-\bar{\phi}\Bigl(\frac{y-\Delta^{-1}_S(t_R)}{100^{-1}\sin^{-1}(1/R)}\Bigr)\Bigr] \Bigl(\frac{100t_R}{R\Delta^{-1}_S(t_R)}\Bigr)+c\bar{\phi}\Bigl(\frac{y-\Delta^{-1}_S(t_R)}{100^{-1}\sin^{-1}(1/R)}\Bigr)\,dy, 
$$
where $c$ is the unique value which ensures that $\theta_1(100\pi)=100\pi.$
For $x \in (0,1)$, we need $\theta_x$ to give a smooth homotopy between $\theta_0$ and $\theta_1.$ We set $$\theta_x(u)=(1-\phi(x))\theta_0+\phi(x)\theta_1.$$


Now form the composition $\Lambda_x:=\bar{\Delta}_S \circ\bar{\theta}^{-1}_x:S^n \to \hat{\sigma}(h,\cdot)_S.$ Pulling back the metric on $\hat{\sigma}(h,\cdot)_S$ via $\Lambda_x$ then gives a one-parameter family of metrics in $\rdok(S^n)$, starting with $\sigma(h,\cdot)_S$ and ending with a metric which agrees with $h$ on the complement of $D_{x_0}.$ Let us denote this last metric by $h'$.
\smallskip

\noindent{\it Step 4.}
The final step (ignoring the basepoint issue) is to smoothly `unwrinkle' the metric $h'$ in $D_{x_0},$ so we end up with the round metric of radius $R$ is this disc. Of course $(D_{x_0},h')$ is isometric to the complement of the $h$-part of $\hat{\sigma}(h,\cdot)_S$ via the restriction of $\Lambda_1.$ The idea here is essentially the same as for contracting the extended $e$-part of $\hat{\sigma}(h,e)$ as carried out in Step 1: we work with the Riemannian manifold $\hat{\sigma}(h,\cdot)_S$, and use the warped product deformation procedure to contract the docking station part of the manifold back to the $h$-part, so we end up with $(S^n,h).$ In order to apply Lemma \ref{blemish} and Remark \ref{start}, it will be convenient to introduce a new parameter $v$ into $\hat{\sigma}(h,\cdot)_S$, given by $v=t-R\pi/2.$ The metric on the complement of the $h$-part of $\hat{\sigma}(h,\cdot)_S$ is then given by $dv^2+\bar{\mu}^2(v)ds^2_{n-1}$ for some function $\bar{\mu},$ which satisfies $\bar{\mu}(v)=R\cos(v/R)$ for $v \in [\tau_R,\kappa_R \tau_R].$ We then apply  Lemma \ref{blemish} and Remark \ref{start} to this metric, setting $\epsilon=1,$ $L=L'-R\pi/2$ and $L_0=L'-R\pi/2-50\pi.$ 

Of course, at each stage of this deformation, we need to exhibit a diffeomorphism of the manifold with the standard sphere, so that when we pull the metric back to standard sphere we obtain the desired path of metrics linking $(S^n,h')$ to $(S^n,h).$ As in Step 1, these diffeomorphisms will be given by Definition \ref{Delta} (extended by the identity over the complement), with existence guaranteed by Lemma \ref{contraction}.

In $\Lambda_1$ we have an isometry between $(S^n ,h')$ and $\hat{\sigma}(h,\cdot)_S.$ The only thing preventing us from immediately implementing the strategy from Step 1 is that restricted to $D_{x_0},$ $\Lambda_1$ does not agree with the contraction diffeomorphism $\bar{\Delta}(100,R).$ (Note that connected-sum contraction map which pulls the complement of the $h$-part of $\hat{\sigma}(h,\cdot)_S$ back to $D_{x_0}$ is precisely $\bar{\Delta}(100,R).$) However we claim that there is a smooth homotopy of diffeomorphisms from $\Lambda_1|_{D_{x_0}}$ to $\bar{\Delta}(100,R)$ (which can be trivially extended by the identity over the complement of $D_{x_0}$ to give a homotopy of global diffeomorphisms $S^n \to \hat{\sigma}(h,\cdot)_S$). Performing this homotopy {\it before} executing the warped product deformation procedure then allows us to apply the arguments of Step 1 to deform $(S^n,h')$ into $(S^n,h).$ 

Thinking in terms of the $t$ and $u$ parameters, we note that each of these diffeomorphisms is determined by a diffeomorphism $[100t_R/R,100\pi] \to [t_R,L'],$ which we will denote $\Gamma_1$ respectively $\Gamma_2$. We have a homotopy $$(1-\phi(x))\Gamma_1(u)+\phi(x)\Gamma_2(u)$$ for $x \in [0,1]$, and this is easily seen to be a homotopy through diffeomorphisms (since being a diffeomorphism here is equivalent to the $u$-derivative being positive). This clearly induces the desired homotopy, and thus establishes the claim.

At this point we have succeeded in deforming $\sigma(h,e)$ to $(S^n,h)$ through metrics of 2-positive Ricci curvature. However, in Step 4 we neglected to consider the basepoint and its neighbourhood, so we cannot guarantee that this last part of the path of metrics stays within $\rdok(S^n).$ Our final step addresses this point.
\smallskip

\noindent{\it Step 5.}
In the metric deformation of Step 4, it is clear that a neighbourhood of the basepoint will be round at every stage, but it is not clear what the radius of these round discs might be, nor if they will always contain a distance sphere with intrinsic radius one. To fix this problem, we can simply scale. For a given metric $h$, or rather, for a given $R>1$, there is clearly a number $\beta\ge 1$ such that scaling the path of metrics in Step 4 by $\beta^2$ keeps the path within $\rdok(S^n)$. For such a choice of $\beta$, we begin the final stage of the deformation by smoothly scaling the metric $h'$ up to $\beta^2h'$, say via the path $\bigl(1+\phi(x)(\beta^2-1)\bigr)h'$ for $x \in [0,1]$. We then follow this by the scaled path from $\beta^2h'$ to $\beta^2h$, and finally reverse the scaling to end with the metric $h$. As indicated above, the choice of $\beta$ will depend on $R$, but is otherwise independent of the metric $h$. Provided we can make a choice of $\beta$ smoothly dependent on $R$, the proof will be complete.

At each stage of the warped product deformation in Step 4 linking $\hat{\sigma}(h,\cdot)_S$ to $(S^n,h)$, we have a manifold containing a warped product disc with metric $dv^2+\gamma_s^2(v)ds^2_{n-1}.$ Here $v$ (defined in Step 4) belongs to an interval $ [0,L(s)]$, with $s$ indexing the deformation. In accordance with Lemma \ref{blemish}, $s$ belongs to the interval $[\tau_R,L'-R\pi/2-50\pi]$. The centre of this disc (i.e. a neigbourhood of the basepoint) is round of some radius depending smoothly on $s$, and thus there is a number $m(s)$ which is the maximal intrinsic radius of a distance sphere within this round neighbourhood. By the compactness of the interval $[\tau_R,L'-R\pi/2-50\pi]$ there is a positive minimum value of $m(s).$ Call this $M$. Clearly $M=M(R)$, with (at least) a continuous dependence on $R$. Choose a smooth function $\bar{\beta}(R)$ with $0<\bar{\beta}(R)<M(R)$ for all $R>1$. Setting $\beta(R)=1/\bar{\beta}(R)$ will then clearly suffice for our scaled deformation to stay within the space $\rdok(S^n)$, as required.
\end{proof}

\begin{proposition}
\label{assoc}
The operation $\sigma$ is homotopy associative.
\end{proposition}


\begin{figure}[!htbp]
\vspace{4cm}
\hspace{-0.5cm}
\begin{picture}(0,0)
\includegraphics{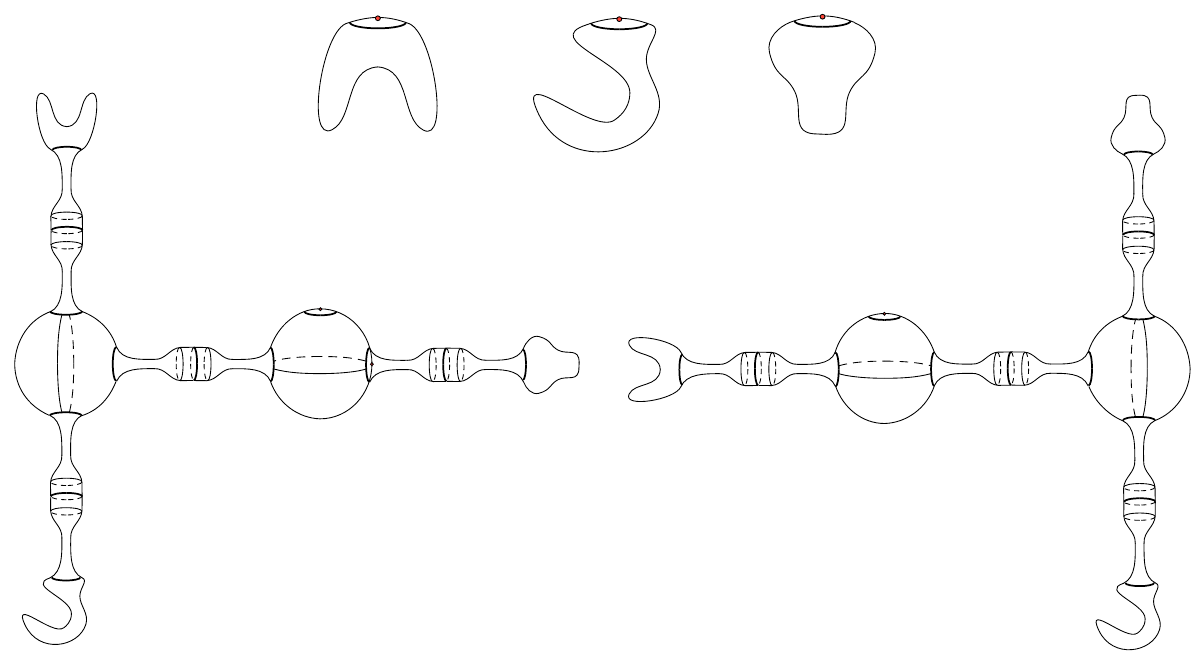}%
\end{picture}
\setlength{\unitlength}{3947sp}
\begingroup\makeatletter\ifx\SetFigFont\undefined%
\gdef\SetFigFont#1#2#3#4#5{%
  \reset@font\fontsize{#1}{#2pt}%
  \fontfamily{#3}\fontseries{#4}\fontshape{#5}%
  \selectfont}%
\fi\endgroup%
\begin{picture}(5079,1559)(1902,-7227)

\put(3600,-4900){\makebox(0,0)[lb]{\smash{{\SetFigFont{10}{8}{\rmdefault}{\mddefault}{\updefault}{\color[rgb]{0,0,0}$h_1$}%
}}}}
\put(4700,-4900){\makebox(0,0)[lb]{\smash{{\SetFigFont{10}{8}{\rmdefault}{\mddefault}{\updefault}{\color[rgb]{0,0,0}$h_2$}%
}}}}
\put(5700,-4900){\makebox(0,0)[lb]{\smash{{\SetFigFont{10}{8}{\rmdefault}{\mddefault}{\updefault}{\color[rgb]{0,0,0}$h_3$}%
}}}}
\put(2350,-4600){\makebox(0,0)[lb]{\smash{{\SetFigFont{10}{8}{\rmdefault}{\mddefault}{\updefault}{\color[rgb]{0,0,0}$\bar{h}_1$}%
}}}}
\put(2350,-7000){\makebox(0,0)[lb]{\smash{{\SetFigFont{10}{8}{\rmdefault}{\mddefault}{\updefault}{\color[rgb]{0,0,0}$\bar{h}_2$}%
}}}}
\put(4400,-6100){\makebox(0,0)[lb]{\smash{{\SetFigFont{10}{8}{\rmdefault}{\mddefault}{\updefault}{\color[rgb]{0,0,0}$\bar{h}_3$}%
}}}}
\put(5000,-6100){\makebox(0,0)[lb]{\smash{{\SetFigFont{10}{8}{\rmdefault}{\mddefault}{\updefault}{\color[rgb]{0,0,0}$\bar{h}_1$}%
}}}}
\put(6900,-7000){\makebox(0,0)[lb]{\smash{{\SetFigFont{10}{8}{\rmdefault}{\mddefault}{\updefault}{\color[rgb]{0,0,0}$\bar{h}_2$}%
}}}}
\put(6900,-4600){\makebox(0,0)[lb]{\smash{{\SetFigFont{10}{8}{\rmdefault}{\mddefault}{\updefault}{\color[rgb]{0,0,0}$\bar{h}_3$}%
}}}}

\put(2600,-6400){\makebox(0,0)[lb]{\smash{{\SetFigFont{10}{8}{\rmdefault}{\mddefault}{\updefault}{\color[rgb]{0,0,0}$\hat{\sigma}(\hat{\sigma}(h_1,h_2), h_3)$}%
}}}}
\put(5800,-6400){\makebox(0,0)[lb]{\smash{{\SetFigFont{10}{8}{\rmdefault}{\mddefault}{\updefault}{\color[rgb]{0,0,0}$\hat{\sigma}(h_1, \hat{\sigma}(h_2,h_3))$}%
}}}}

\end{picture}%
\caption{The metrics $h_1, h_2, h_3, \hat{\sigma}(\hat{\sigma}(h_1,h_2), h_3)$ and $\hat{\sigma}(h_1, \hat{\sigma}(h_2,h_3))$}
\label{Associativity}
\end{figure}
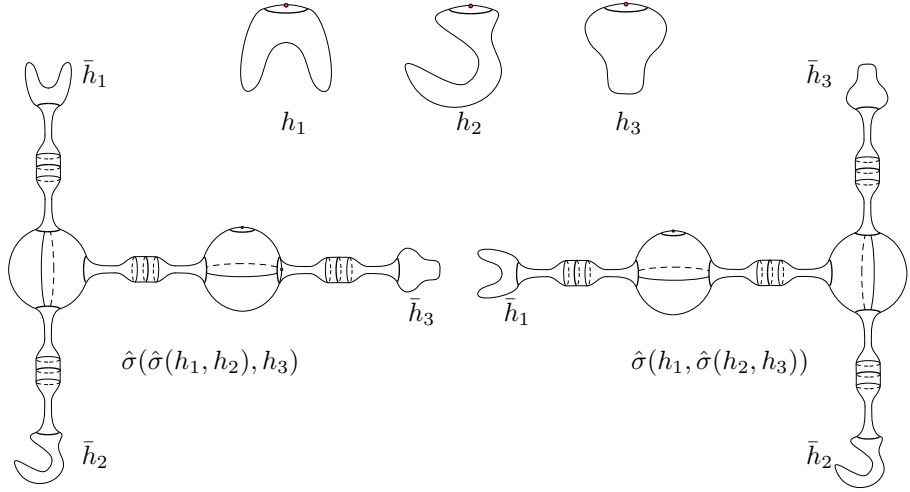

\begin{proof} Consider the product $\sigma(\sigma(h_1,h_2),h_3)$ for $h_i \in \rdok(S^n).$ We must show that up to homotopy, this is the same as $\sigma(h_1,\sigma(h_2,h_3)).$ Let $\hat{\sigma}(\hat{\sigma}(h_1,h_2),h_3)$ and $\hat{\sigma}(h_1,\hat{\sigma}(h_2,h_3))$ denote the obvious connected sum arrangements corresponding to these triple products (see Definition \ref{H_mult}). These are depicted in Fig. \ref{Associativity} above.

We begin by noting that there are canonical paths of metrics through $\rdok(S^n)$ from $\sigma(h_1,\sigma(h_2,h_3))$ to $\sigma(\sigma(h_2,h_3),h_1),$ and from $\sigma(\sigma(h_2,h_3),h_1)$ to $\sigma(\sigma(h_3,h_2),h_1),$ given by rotation (see Proposition \ref{H_commute}). Thus it will suffice to show that up to homotopy, $\sigma(\sigma(h_1,h_2),h_3)$ agrees with $\sigma(\sigma(h_3,h_2),h_1).$ 

We define a smooth manifold $\mathcal D$ as follows. Take two copies $S_a$ and $S_b$ of the Riemannian manifold $(S^n,100^2ds^2_n)$, and let $D_a$ (respectively $D_b$) denote the disc about the basepoint of $S_a$ (respectively $S_b$) for which the boundary has intrinsic radius 1. Now form the connected sum $$(S_a\setminus D_a) \cup {\mathcal T}_{100} \cup C(\rho_{100},\rho_{100}) \cup {\mathcal T}_{100} \cup (S_b \setminus D_b),$$ and let $\mathcal D$ denote the underlying smooth manifold (ignoring the metric). 
Notice that we can view both $\hat{\sigma}(\hat{\sigma}(h_1,h_2),h_3),$ $\hat{\sigma}(\hat{\sigma}(h_3,h_2),h_1)$ as connected sums involving $\mathcal D$, with $\mathcal D$ in effect playing the role of a double docking station. We will declare the basepoint $x_0$ of $\mathcal D$ to be a point on the equator of the $S_b$ sphere, (which is consistent with the basepoint locations on $\hat{\sigma}(\hat{\sigma}(h_1,h_2),h_3)$ and $\hat{\sigma}(\hat{\sigma}(h_3,h_2),h_1)$). 

Equipping $\mathcal D$ with its natural metric $g_{nat}$ from the above connected sum arrangement, we obtain via Lemma \ref{contraction} a smooth map $$\bar{\Delta}_{\mathcal D}:D_b \to {\mathcal D}\setminus (S_b \setminus D_b).$$ We can extend this by the identity on the complement of $D_b$ in $S_b$ to obtain a diffeomorphism $\hat{\Delta}_{\mathcal D}:S_b \to {\mathcal D}.$

Now consider any metric $g_{\mathcal D}\in \rdok(\mathcal D)$. Given three points $z_1,z_2,z_3$ on $\mathcal D$, suppose that $g_{\mathcal D}$ is round of some radius $\bar{R}_i$ in a neighbourhood of $z_i$, and is such that each round neighbourhood contains a sphere about $z_i$ with intrinsic radius 1. Let $D_i$ denote the disc centred on $z_i$ bounded by this sphere, and suppose that the $D_i$ are disjoint from each other, and from the corresponding disc $D(x_0)$ about the basepoint. Consider metrics $h_i \in \rdok(S^n)$, $i=1,2,3$, locally round of radius $R_i$ in a neighbourhood of the basepoint. Form the Riemannian manifolds $$(S^n\setminus D^n,h_i) \cup {\mathcal T}_{R_i}\cup C(\rho_{R_i},\rho_{\bar{R}_i}).$$ By Lemma \ref{contraction}, the above manifolds come with contracting diffeomorphisms $\bar{\Delta}_i:=\bar{\Delta}(R_i,\bar{R}_i)$ from the disc $D_i$ in $\mathcal D$ to the corresponding connected sum. Pulling back the Riemannian metric via $\bar{\Delta}_i$ for $i=1,2,3$ then gives a new Riemannian metric on $\mathcal D$. Subsequently pulling-back by $\hat{\Delta}_{\mathcal D}$ to $S_2$ then gives a metric in $\rdok(S^n).$

The upshot of the above is that any Riemannian connected sum arrangement of the type described above gives a metric in $\rdok(S^n)$, in a fixed way which varies smoothly with the `input' metrics $h_1,h_2,h_3$ and $g_{\mathcal D}$.

Observe that if we apply the above double-contraction procedure to the connected sum arrangement $\hat{\sigma}(\hat{\sigma}(h_1,h_2),h_3)$ or $\hat{\sigma}(\hat{\sigma}(h_3,h_2),h_1)$, then we obtain precisely $\sigma(\sigma(h_1,h_2),h_3)$ respectively $\sigma(\sigma(h_3,h_2),h_1).$ (Note that in either connected sum arrangement, the metric restricted to the `body' of $\mathcal D$ is $g_{nat}$.) It therefore suffices to find a smooth one-parameter family of $Ric_k>0$ connected-sum arrangements linking $\hat{\sigma}(\hat{\sigma}(h_1,h_2),h_3)$ with $\hat{\sigma}(\hat{\sigma}(h_3,h_2),h_1)$, which is independent of the metrics $h_1,h_2,h_3.$

We start with $\hat{\sigma}(\hat{\sigma}(h_1,h_2),h_3)$. Corresponding to each $h_i$ in this arrangement is a point $z_i \in \mathcal D$, as above. Let us choose and fix two smooth paths in $\mathcal D$: $\gamma_1(t)$ will be a path from $z_1$ to $z_3$; $\gamma_3(t)$ will be a path from $z_3$ to $z_1.$ In both cases we will suppose that $t \in [0,1],$ and we will assume that the closed disc of radius $200\sin^{-1}(1/100)$ about $\gamma_i(t)$ does not contain the basepoint $x_0$, the point $z_2$, nor the point $\gamma_j(t)$, for $i\neq j$. Roughy speaking, the idea is to slide the disc $$(S^n\setminus D^n,h_i) \cup {\mathcal T}_{R_i} \cup C(\rho_{R_i},\rho_{100}) \cup {\mathcal T}_{100}$$ along the path $\gamma_i$, $i=1,3,$ which will result in the desired swapping of the $h_1$ and $h_3$ arrangements. The separation of the paths from each other and the point $z_2$ ensures that the extended $h_i$ parts do not intersect during this sliding deformation. The distance of the paths from $x_0$ ensures that the metrics we construct will always pull back to $\rdok(S^n).$

The difficulty with carrying out this sliding deformation is that the curvature of $({\mathcal D},g_{nat})$ along the paths $\gamma_i(t)$ will vary, and in particular there will be points along these paths which will not have a locally round neighbourhood. As it stands, therefore, we cannot attach the `tubes' corresponding to $h_1$ and $h_3$ at some points along $\gamma_1,\gamma_3.$ To deal with this situation we apply Lemma \ref{basepoint} with $k=2$ to both paths $\gamma_i(t)$ within $({\mathcal D},g_{nat}).$ This results in a smooth path of metrics $g(t)$ in $\mathcal{R}^{Ric_2>0}_{rd,1}(\mathcal{D})$ which locally, in a neighbourhood of $\gamma_i(t),$ is round of some radius $\bar{R}_i(t).$ There is no guarantee, however, that this round region will contain a distance sphere about $\gamma_i(t)$ with intrinsic radius 1. On the other hand, as this is a compact family of metrics, the functions $\bar{R}_i(t)$ are clearly bounded below away from zero. In particular, there is a scaling factor $\beta\ge 1$ such that the path $\beta^2 g(t)$ {\it does} have the property that for each $t$, the round neighbourhood in $\beta^2g(t)$ about the point $\gamma_i(t)$ contains a distance sphere with intrinsic radius 1. This allows us to attach the tube $T_{\beta \bar{R}_i(t)}$ about the point $\gamma_i(t)$ for each $t$, and to the narrow end of this tube we can attach $(S^n\setminus D^n,h_i)\cup {\mathcal T}_{R_i} \cup C(\rho_{R_i},\rho_{\beta \bar{R}_i(t)}).$ Notice that we can choose the value of $\beta$ independent of $h_1,h_2,h_3$, as it only depends on the choice of paths $\gamma_i$ and the metric $g_{nat}$.

To slide the metrics $h_1$ and $h_3$ as required while taking into account the necessary scaling by $\beta^2$, we proceed as follows. The very first task is to smoothly scale the metric on $\mathcal D$ by a progressive scaling factor starting with 1 and ending with $\beta^2$. We could do this, for example, using the bump function $\phi(x)$ for $x \in [0,1]$ and applying the scaling factor $\beta^2(x):=1+\phi(x)(\beta^2-1).$ When scaling $\mathcal D$, for each $x \in [0,1]$ we replace the attaching tubes for $h_1,h_2,h_3$ in $\mathcal D$ by $$(S^n\setminus D^n,h_j) \cup {\mathcal T}_{R_j} \cup C(\rho_{R_j},\rho_{100\beta(x)})\cup {\mathcal T}_{100\beta(x)},$$ for $j=1,2,3,$ which gives a smooth deformation. Having achieved an appropriate scaling for $\mathcal D$, we can now move the $h_i$ part of the arrangement, $i=1,3,$ along the path $\gamma_i(t)$, as indicated in the above paragraph. After completing the path, we now simply reverse the scaling of $\mathcal D$, so the resulting Riemannian manifold is $\hat{\sigma}(\hat{\sigma}(h_3,h_2),h_1)$, as desired.

\end{proof}

\begin{proof}[Proof of Theorem A]
The proof now follows by combining Propositions \ref{H_commute}, \ref{identity} and \ref{assoc}.
\end{proof}

\begin{proof}[Proof of Corollary B]
This is a basic fact about all $H$-spaces; see in particular Corollary 5.2 of \cite{Wa}.
\end{proof}


\section{Loop space structures}\label{LSS}
We turn our attention now to Theorem C. The proof strategy mirrors that of the second main result in \cite{Wa}, adapted here for $k$-positive Ricci curvature. We will therefore be quite terse in our exposition, referring the reader to \cite{Wa} for relevant details.

In the previous section we demonstrated that up to homotopy, the space $\kRP(S^n)$, or rather the homotopy equivalent space $\rdok(S^n)$, admits an $H$-space structure when $n\geq 3$ and $k\geq 2$. The question of whether an $H$-space admits a loop space or iterated loop space structure is an old one, and various recognition results for deciding when this happens have been established by Stasheff, Boardman, Vogt, May and others. In particular, if we can generalize the homotopy product to an action of a certain operad, we go a long way to exhibiting iterated loop space structure. Just as in the previous section, our approach here is to demonstrate the existence of an iterated loop space structure on ${\mathcal R}^{Ric_k>0}(S^n)$ by showing that $\rdok(S^n)$ admits such a structure.

\subsection {The operad of little discs and the bar construction}
An {\em operad} $\mathcal{P}$ consists of a sequence of compactly generated Hausdorff topological spaces $\mathcal{P}(j)$, $j\in\{0,1,2,\cdots\}$, together with data consisting of various maps and labels which satisfy certain symmetry and composition conditions. The full definition is rather complicated and so we will not provide it here. Instead we will briefly recall an important example: {\em the operad of little discs}. This example (and others) is described in detail, along with a complete definition of the term operad (due to P. May), in section 7 of \cite{Wa}. 

\subsection*{The operad of little $n$-dimensional discs}
Let $D^n$ denote here the closed unit radius disc in ${\mathbb R}^n.$ For $n\geq 1$, $p\in D^{n}$ and $\epsilon$ where $0<\epsilon\leq 1-|p|$, let $D(p,\epsilon)$ denote the closed round disc of radius $\epsilon$ which is centred at $p$. For each integer $j\geq 0$, we denote by $\D(j)_{n}$ the set of ordered $j$-tuples of closed round discs $D(p_i, \epsilon_i)$, where $i=1,\cdots j,$ which satisfy the following condition:
\begin{equation*}
\oD (p_i,\epsilon_i)\cap \oD (p_k,\epsilon_k)=\emptyset {\text{ for all } } i\neq k.
\end{equation*}
In the case when $j=0$, $\D(j)_n$ is just a single point. There is an identity element in $\D(1)$, namely the element for which the unique little disc is equal to the whole of $D^{n}$. To ease notation we will fix an $n$ and simply write $\D(j)$ instead of $\D(j)_{n}$. There is an obvious action of the permutation group $\Sigma_j$ on $\D(j)$ which essentially permutes the labels of the little discs; this is described in section 7 of \cite{Wa}. Notice that for each little disc in an element of $\D(j)$, there is a canonical homeomorphism which identifies it with the larger unit disc $D^{n}$, i.e. shrink $D^{n}$ and translate. This leads to the following ``fitting" map:
\begin{equation*}
\begin{split}
\gamma:\D(k)\times \D(j_1)\times\cdots\times \D(j_k)&\longrightarrow \D(j_1+\cdots+j_k)\\
(c,(d_{j_1},\cdots, d_{j_k}))&\longmapsto c(d_{j_1},\cdots, d_{j_k}),
\end{split}
\end{equation*}
which replaces the $r^{th}$ little disc of $c$ with the appropriately rescaled element $d_{j_r}$ for each $r\in\{1,2,\cdots, k\}.$
\noindent Finally, we define 
\begin{equation*}
\D_n=\bigcup_{j\geq 0}\D(j),
\end{equation*}
which, along with the appropriate collection of fitting maps $\gamma$, is known as the {\em operad of little $n$-dimensional discs}.

In section 7.6 of \cite{Wa}, we describe what it means for an operad, $\mathcal{P}$, to {\em act} on a topological space $Z$. We will not repeat the definition here, although we will construct an example of such an action later. A space $Z$ admitting an action of the operad $\mathcal{P}$ is called a {\em $\mathcal{P}$-space}. We now state a classic theorem, due to Boardman, Vogt and May, the so-called {\em recognition principle} for iterated loop spaces.
\begin{theorem}[Boardman and Vogt \cite{BV}, May \cite{May}]\label{BVM}
For any $n\in \mathbb{N}$, a path connected ${\D}_n$-space, $Z$, is weakly homotopy equivalent to an $n$-fold loop space.
\end{theorem}
\begin{remark}There is a version of this theorem for non-path connected spaces but it requires another hypothesis, namely that $\pi_{0}(Z)$ forms a group under multiplication induced by the operad action. We will say a few words about this at the end of the paper in Section \ref{grouplike}.
\end{remark}
The above theorem suggests a means for proving Theorem C: simply demonstrate an action of $\D_{n}$ on $\rdok(S^{n})$. Such an action restricts as a perfectly good operad action on the path component containing the round metric, and Theorem \ref{BVM} does the rest. Morally, this is what we do, however there are some technical difficulties when it comes to constructing such a $\D_n$-action. Roughly speaking, certain commutativity requirements are impossible to satisfy. To overcome this, we replace $\D_{n}$ with a more flexible related operad known as $W\D_{n}$. This latter operad is obtained by by applying something called the {\em bar construction} (or {\em W construction}) to $\D_{n}$, a process which is explained in section 7.5 of \cite{Wa}, and of which we will give a very brief description in a moment.  Essentially, attempting to construct an action (of $\D_{n}$) yields something which satisfies the commutativity relations only up to homotopy. This discrepancy is absorbed into the construction of $W\D_{n}$, in order to make the resulting action commute ``on the nose." Given that our entire goal is the study an object only up to homotopy, it is not surprising that we can make such a replacement. This is confirmed by the following theorem of Boardman and Vogt.
\begin{theorem}[\cite{BV},Theorem 4.37]\label{BVthm}
A topological space $Z$ is a $\mathcal{P}$-space, for some operad $\mathcal{P}$, if and only it is a $W\mathcal{P}$-space.
\end{theorem}
\noindent Thus, to prove Theorem C, it is enough to proceed as in \cite{Wa} and demonstrate an action of $W\D_{n}$ on $\rdok(S^n)$.

While sections 7.3 through 7.5 of \cite{Wa} provide a detailed description of the bar construction, we will say a few words here about the operad $W\D_n$. As we are assuming that $n$ is fixed, let us simplify notation by writing $\D$ instead of $\D_n$. The construction of $W\D$ consists of two main steps. The first is the construction from $\D$ of an intermediary operad $\T \D$, the so-called {\em operad of trees for $\D$}. The second step realises $W\D$ as a quotient of $\T\D$ by certain relations. We now provide a brief outline of these steps.

The operad $\T\D$ should be understood as follows. As for all operads, $\T\D(0)$ is a single point. Roughly speaking, an element of $\T\D(j)$, for some $j\in\mathbb{N},$ consists of a certain oriented tree, $T$, with some associated data which we will deal with shortly. Edges of $T$ which have two adjacent vertices are called {\em internal}, while edges with only one adjacent vertex are called {\em external}. Each vertex $v$ in $T$ has a set of incoming edges, denoted $\In(v)$, and exactly one outgoing edge. 
Moreover, the set of external edges of $T$ consists of two mutually disjoint subsets: the set of {\em inputs}, $\In(T)$, of all incoming edges of $T$ which have no starting vertices, and the set consisting of the single outgoing edge, or {\em output}, which has no end vertex. The statement that $T$ belongs to $\T\D(j)$ means that $T$ has $j$ external input edges, i.e. $|\In(T)|=j$.

The extra associated data that comes with the tree $T$ is as follows. Firstly, there is a labelling of the input edges via a function $\beta$, a bijection from the set of inputs $\In(T)$ to $\{1,2,\cdots ,j\}$. Secondly, there is a map, $\alpha$, which sends each vertex $v$ of $T$ to an element $c\in \D(|\In(v)|)$, so $c$ has one (labelled) little disc for each (labelled) input edge of $v$.  Finally, there is a function $\ell$ which assigns to each edge $e$ of $T$ a real number $\ell(e)$ (the so-called {\em length} of the edge) satisfying:
\begin{enumerate}
\item{}$0\leq \ell (e) \leq 1$.
\item{}$\ell (e)=1$ when $e$ is an external edge (i.e. input or output) of $T$.
\end{enumerate}
When the edge under discussion is clear from context, we will simply write $\ell$ for its length.
With all this in mind, an element of $\T \D$ should be thought of as a quadruple $(T, \alpha, \beta, \ell)$. In depicting such an element, we typically show trees with the edges directed from bottom to top and with inputs ordered from left to right. In Fig. \ref{treeoperad} 
below we present an example with four input edges. The internal edge lengths arising from $\ell$, along with the labellings on the little discs and input edges are suppressed in the picture. The association between little disks at a vertex and the input edges of the vertex should be clear from the picture however.

\begin{figure}[!htbp]
\vspace{4cm}
\hspace{3cm}
\begin{picture}(0,0)
\includegraphics{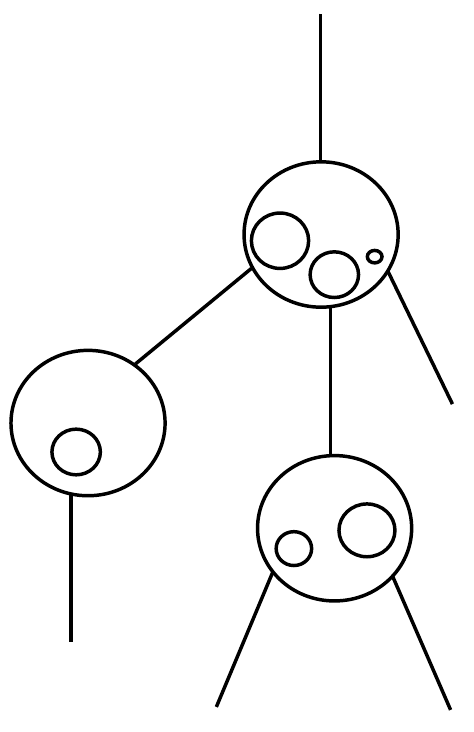}%
\end{picture}
\setlength{\unitlength}{3947sp}%
\begingroup\makeatletter\ifx\SetFigFont\undefined%
\gdef\SetFigFont#1#2#3#4#5{%
  \reset@font\fontsize{#1}{#2pt}%
  \fontfamily{#3}\fontseries{#4}\fontshape{#5}%
  \selectfont}%
\fi\endgroup%
\begin{picture}(5079,1559)(1902,-7227)
\end{picture}%
\caption{An element of $\T\D(4)$}
\label{treeoperad}
\end{figure} 

There are also certain associated composition (fitting) maps which we will not define here except to say the following. An element $c\in\T\D(k)$ can be composed with a string of elements $(c_{1}, \cdots, c_{k})\in\T\D(j_1)\times\cdots\times \T\D(j_k)$ by laying the output edge of each $c_i$ directly on top of the $i^{th}$ input edge of $c$. In each case, the edge resulting from this superposition is given the length $1$ (agreeing with the length of the superposed edges). 
Finally, the identity element in $\T\D(1)$ is the {\em trivial tree} consisting of a single edge (with length $1$) and no vertices.

Having constructed $\T \D$, the operad $W\D$ is then obtained as a quotient of $\T \D$ by certain relations, namely (a), (b) and (c) of section 7.5 of \cite{Wa}. We will not properly restate these relations here, except to say the following.
\begin{enumerate}
\item[(a)] Identifies two edges seperated by a vertex decorated by the identity element in $\D$ with a single edge. The length of the resulting edge is set to be $\ell_1+\ell_2-\ell_1\ell_2$, where $\ell_1,\ell_2$ are the lengths of the original edges.
\item[(b)] Identifies trees up to certain label permutations on subtrees.
\item[(c)] Collapses edges of length zero by using the operad composition maps to combine the operad elements associated to their
vertices.
\end{enumerate}
Thus,
$$W\D:=\T\D/\mathrm {Relations } \{(a), (b), (c)\}.$$

\subsection{Realizing the operad action}

Given any tree of the type under discussion with vertices decorated by little discs, our next task is to construct a metric in ${\mathcal R}^{Ric_2>0}_{rd,1}(S^n)$ which reflects both the structure of the tree and the vertex details. Once we have constructed such metrics, we will use them to produce an action of the operad $W\D$ on $\rdok(S^n)$. Theorem C then follows from the existence of such an action.

Our approach to metric construction here mirrors that of the previous section. Given a tree $T$, we will construct a Riemannian manifold $\hat{\sigma}(T)$ as a connected sum of spheres. We then use the connected sum contraction procedure to distill a metric $\sigma(T)\in{\mathcal R}^{Ric_2>0}_{rd,1}(S^n)$ from $\hat{\sigma}(T)$. The various pieces of this connected sum will arise from the edges of the tree, with the little discs at each vertex and the edge lengths feeding into the geometric structure.

We will begin by discussing the `basic unit' of  $\hat{\sigma}(T)$, which corresdponds to an edge of $T$. This will be a disc determined by three parameters: $R,\epsilon,\ell,$ (with $R>0,$ $\epsilon \in (0,R)$ and $\ell \in [0,1]$). Here, $R$ and $\epsilon$ determine the geometric boundary conditions of the disc: just as in previous considerations, a neighbourhood of the boundary will be round of radius $R$, with the boundary itself being a round sphere of intrinsic radius $\epsilon$. As indicated above, $\ell$ will indicate the `length' of the unit (i.e. the radius of the disc), however $\ell$ should be viewed as a {\it proportion}, not an absolute length. With the dimension understood, let us denote this basic unit $B(R,\epsilon,\ell).$ In the case where $\ell=1,$ we simply write $B(R,\epsilon).$

We begin by defining $B(R,\epsilon)$ as follows: $$B(R,\epsilon):= {\mathcal T}_{R,\epsilon} \cup C(\rho_{R,\epsilon},\rho_{100}) \cup (S^n\setminus D^n,100^2ds^2_n).$$ We will introduce a parameter $t$ into this warped product manifold as before, with $t$ measuring distance from the boundary, and the boundary itself corresponding to $t=\tau_{R,\epsilon}=R\cos^{-1}(\epsilon/R).$ The domain of $t$ is $[\tau_{R,\epsilon},L]$, for some $L$ depending on $R,\epsilon.$

We will now explain how to modify $B(R,\epsilon)$ to take account of a value of $\ell \neq 1.$ For this we recall Lemma \ref{blemish}. Taking $\mu(t)$ in this Lemma to be the warping function for the metric on $B(R,\epsilon),$ let $\gamma_s$ be the function established by the Lemma, with $s$ here belonging to the interval $[\tau_{R,\epsilon},L-50\pi].$  Given $\ell \in [0,1],$ define $B(R,\epsilon,\ell)$ to be the disc equipped with the warped product metric $$dt^2+\gamma^2_{\chi(\ell)}(t)ds^2_{n-1},$$ where $$\chi(\ell)=\tau_{R,\epsilon}+\ell(L-50\pi-\tau_{R,\epsilon}).$$
It is clear that when $\ell=1,$ $B(R,\epsilon,\ell)=B(R,\epsilon)$, and it follows from Lemma \ref{blemish} that when $\ell\approx 0$, $B(R,\epsilon,\ell)$ is a disc equipped with the round metric of radius $R$ bounded by the sphere of intrinsic radius $\epsilon$. It follows that relation (c) above will be satisfied by this construction. By Observation \ref{observation}, the central neighbourhood of $B(R,\epsilon,\ell)$ corresponding to $t \in [\chi(\ell)+\alpha(R,\epsilon),L(\chi(\ell))]$ is round, with both the curvature and the radius of this neighbourhood varying smoothly with $\ell$. (Recall that $\alpha(R,\epsilon)$ arises due to smoothing considerations - something we must also be mindful of here.)


We next address the role of little discs in our construction. Suppose an edge corresponding to $B(R,\epsilon,\ell)$ ends with a vertex coloured by an element of ${\mathcal D}(j)$. This element consists of the standard unit ($n$-dimensional) disc, containing a collection of $j$ smaller discs. For convenience we will refer to these as the `big disc' and `little discs' respectively. 

We can identify the big disc and the round region of $B(R,\epsilon,\ell)$ (corresponding to $t \in [\chi(\ell)+\alpha(R,\epsilon),L(\chi(\ell))]$) in a standard way: we begin by stretching the big disc by a uniform factor so as to have the same radius as the round region. Next we equip this scaled disc with the metric $dt^2+R'^2\sin^2(t/R')ds^2_{n-1},$ where $1/R'^2$ is the curvature of the round region in $B(R,\epsilon,\ell).$ This has no effect on distances of points from centre, but causes distortion in directions orthogonal to the radii. We can now identify the two discs via a fixed isometry $\iota$. We will also need to `transfer' the little discs to $B(R,\epsilon,\ell).$ Each little disc $D(p_i,r_i)$ in the original big disc is determined its centre point $p_i$ and its radius $r_i$. Suppose that the stretching factor for the big disc is $\lambda>0$. Then the stretching effect on the point $p_i$ in the Euclidean setting is to move it to the point $\lambda p_i.$ It then follows (for example from the Toponogov Theorem) that the collection of discs $D(\iota(\lambda p_i),\lambda r_i)$ within the round region of $B(R,\epsilon,\ell)$ are disjoint, and hence form an admissible family of little discs in the operad sense. In this way we can map both the big disc and its little discs onto the round region of $B(R,\epsilon,\ell).$ Notice that the geometry of each little disc is determined by a pair $R',\epsilon'$, with $R'$ as above, and $\epsilon'$ being the intrinsic radius of the sphere bounding the little disc.

As noted previously, when $\ell=0$, $B(R,\epsilon,0)$ is a round disc of radius $R$ bounded by the sphere of intrinsic radius $\epsilon$. For convenience we will demand that the big disc in $B(R,\epsilon,0)$ agrees with $B(R,\epsilon,0)$ itself. This forces us into a slight modification of the big disc identification just introduced, since as it stands, the boundary of the big disc when $\ell=0$ corresponds to $t=\tau_{R,\epsilon}+\alpha(R,\epsilon),$ and not to the desired $t=\tau(R,\epsilon).$ Recall that by the proof of Lemma \ref{blemish}, $\gamma_s$ agrees with $R\cos(t/R)$ for $\tau_{R,\epsilon} \le s \le \tau_{R,\epsilon}(\kappa_{R,\epsilon}+1)/2.$ Thus for $s$ in this range, the `smoothing buffer' of length $2\alpha(R,\epsilon)$ is redundant. In our current context, $s=\chi(\ell),$ and we deduce that $B(R,\epsilon,\ell)=B(R,\epsilon,0)$ for $\tau_{R,\epsilon} \le \chi(\ell) \le \tau_{R,\epsilon}(\kappa_{R,\epsilon}+1)/2.$ Thus for $\ell$ such that $\chi(\ell)$ lies in this range, the round region contains the whole of $B(R,\epsilon,0)$. We should therefore expect the radius of the round region to vary discontinuously as $\ell \to 0^+.$ On the other hand, we need the radius of the big disc to increase smoothly with $\ell$ when $\chi(\ell)$ lies in this range, from corresponding to $t \in [\chi(\ell)+\alpha(R,\epsilon),L(\chi(\ell))]$ when $\chi(\ell)=\tau_{R,\epsilon}(\kappa_{R,\epsilon}+1)/2$, to $t \in [\chi(\ell),L(\chi(\ell))]$ when $\chi(\ell)=\tau_{R,\epsilon}.$ To achieve this, we declare that when $\chi(\ell) \in [\tau_{R,\epsilon},\tau_{R,\epsilon}(\kappa_{R,\epsilon}+1)/2],$ the big disc will correspond to $t \in [\chi(\ell)+\xi_{R,\epsilon}(\ell)\alpha(R,\epsilon),L(\chi(\ell))],$ where $$\xi_{R,\epsilon}(\ell):=\phi\Bigr(\frac{\chi(\ell)-\tau_{R,\epsilon}}{\tau_{R,\epsilon}(\kappa_{R,\epsilon}-1)/2}\Bigr).$$

When $\ell=1$, we will demand that the big disc in $B(R,\epsilon)=B(R,\epsilon,1)$ is the hemispherical disc about the centre point corresponding to $t \in [L-50\pi,L]$, (as opposed to $t \in [L-50\pi+\alpha(R,\epsilon),L]$). The point is to render the big disc independent of $R,\epsilon$ when $\ell=1.$ In order to do this in a way which is smooth as $\ell\to 1^-$, we simply employ the strategy in the paragraph above, in conjunction with Remark \ref{start}. (In effect, we scale the $\alpha(R,\epsilon)$-term defining the boundary of the big disc smoothly down to zero, as $s$ increases over the range $[L-50\pi-1,L-50\pi]$.)

We are now in a position to discuss the construction of $\hat{\sigma}(T).$ Given any vertex $v$ in $T$, each edge which meets $v$ is either incoming or outgoing, so we can decorate each edge with an arrow. Overall, this produces a directed tree. As noted previously, to be consistent with the figures, we will orient our trees so that the direction of each edge is upwards on the page. Thus the set of external input edges $\text{In}(T)$ will be at the bottom of the tree, and the single external output edge is at the top. Recall that all the external edges have length one.

We will construct $\hat{\sigma}(T)$ systematically, edge-by-edge, starting at the top of the tree and working downwards, considering all edges at a given `depth' from the top, before moving onto to edges one layer deeper.

To the top (output) edge we will associate the sphere with round metric of radius 100. If the tree is the trivial tree, then we are done. If not, we consider the vextex at the (lower) end of the output edge. We declare the big disc here to be the lower hemisphere, and associate little discs within this big disc as described above. To each of these little discs we have a pair of parameters $R$ and $\epsilon$: for this top vertex we clearly have $R=100$. Taking into account the lengths $\ell$ of the outgoing edges from this vertex, we form a connected sum with the appropriate $B(R,\epsilon,\ell)$ for each little disc. Now continue in this way to the bottom of the tree. Observe that in general, $R$ and $\epsilon$ for each little disc will depend on the length of the previous edge as well as the little disc itself. The result of this process is a Riemannian manifold $\hat{\sigma}(T)$ with 2-positive Ricci curvature. If we declare the north pole of the initial sphere to be the basepoint of the arrangement, then we will always have a hemisphere of (extrinsic) radius 100 about the basepoint irrespective of the tree.

To create the metric $\sigma(T) \in {\mathcal R}^{Ric_2>0}_{rd,1}(S^n)$ from $\hat{\sigma}(T)$, we simply apply the connected sum contraction procedure. We do this systematically in the opposite direction: we start at the bottom and work layer-by-layer up to the top. Notice that within a layer, the order in which we perform the contractions is irrelevant. (We are contracting onto the little discs in the layer immediately above, and by construction these are all disjoint.) Notice that the resulting metric in ${\mathcal R}^{Ric_2>0}_{rd,1}(S^n)$ also has a  hemisphere of radius 100 about the basepoint of $S^n$.

In describing this metric construction, there is one important detail we have so far neglected to mention. In order to accommodate relation (a), we are forced to make a special construction in the case that a vertex of our graph is decorated with an element of ${\mathcal D}(1),$ i.e. a big disc containing a single little disc. If this little disc is equal to the big disc, relation (a) requires that the (unique) incoming edge and the outgoing edge be identified, with the length of the new edge being $\ell_1+\ell_2-\ell_1\ell_2,$ where $\ell_1,\ell_2$ are the lengths of the original edges. We must pay special attention to continuity here: suppose we have a smoothly expanding little disc, with radius $r \to 1^-.$ Relation (a) only applies when $r=1$, but our metric constructions must vary smoothly with respect to $r$. In particular this includes the moment at which $r$ hits 1 and the edges become identified. Our construction in this case is as follows.

Suppose we have an arrangement within our tree where we have a vertex with a single incoming edge of length $\ell_1,$ an outgoing edge of length $\ell_2,$ and a single little disc of radius $r<1.$ When we meet such an arrangment in the construction of $\hat{\sigma}(T)$, we move some of the outgoing length to the incoming edge when creating the manifold. Specifically, we represent the incoming edge by $$B\bigl(R,\epsilon,(1-\psi(r))\ell_1+\psi(r)(\ell_1+\ell_2-\ell_1\ell_2)\bigr)$$ and the outgoing edge by $$B(R',\epsilon',(1-\psi(r))\ell_2),$$ where $\psi:[0,1]\to [0,1]$ is any choice of strictly increasing function with $\psi(0)=0$ and $\psi(1)=1$, which can be extended with constant value 0 in the negative direction and constant value 1 in the positive direction to give a smooth function on all of $\mathbb R$. Notice that $R'$ here depends on the disc associated to the incoming edge, and $\epsilon'$ depends on both $R'$ and $r$.

Now consider what happens in the limit as $r \to 1^-$. Clearly, the length of the incoming edge smoothly converges to $\ell_1+\ell_2-\ell_1\ell_2$, and the length of the outgoing edge to 0. Thus the manifold corresponding to this tree arrangement smoothly deforms into that corresponding to a single edge of length $\ell_1+\ell_2-\ell_1\ell_2$, meaning our construction of $\hat{\sigma}(T)$ (and therefore $\sigma(T)$) is consistent with with relation (a).

(By construction, the big disc at the end of length one edges is always a hemisphere. In the case that the corresponding vertex is decorated by the identity little disc, we cannot apply the Tube Lemma (Lemma \ref{tube}) directly as we have $R=\epsilon$. However this does not matter, since by relation (a) and the construction above, we never need to add a tube in this situation.)

For completeness we should also mention relation (b). This is simply a tree re-labelling relation, and has precisely no implications for our metric constructions. Thus we have, in effect, constructed a well-defined map $W{\mathcal D}_n \to {\mathcal R}^{Ric_2>0}_{rd,1}(S^n).$

There are two further considerations we must address. Firstly, we need to show how to realize the composition operation on trees metrically. We work intially with the manifolds $\hat{\sigma}(T).$ The procedure is as follows: suppose an incoming edge of a tree $T_1$ is to be identified with the outgoing edge of another tree $T_2$ when the trees are composed. Both these edges have length 1, and thus are represented in $\hat{\sigma}(T_i)$ by a connected sum with a radius 100 round sphere. Assuming these trees are non-trivial, there will be further connected sums from the northern hemisphere of the `incoming' sphere in $\hat{\sigma}(T_1)$, and from the southern hemisphere of the `outgoing' sphere in $\hat{\sigma}(T_2)$. This means that we can remove the sourthern hemisphere from the first of these spheres, and the northern hemisphere from the latter, and simply glue the two resulting boundary components together. Notice that the resulting Riemannian manifold agrees with $\hat{\sigma}(T_1 \star T_2),$ where $\star$ denotes the composition operation. It is elementary to observe that this association of Riemannian manifolds to trees is associative with resepect to the composition of trees, as our joining of manifolds does not affect the geometry either higher or lower in the construction.

Of course, we actually need to realize the composition operation on trees in the space ${\mathcal R}^{Ric_2>0}_{rd,1}(S^n),$ and this requires repeated connected sum contractions as discussed above. Associativity is automatic here, as a consequence of the order in which the contractions are performed.

(Note that the special arrangement considered above, involving two edges linked by the identity little disc as vertex, cannot be created by joining two graphs.)

Our final task is to define an operad action: specifically the action of $W{\mathcal D}_n$ on the space $\rdok(S^n).$ (See \cite[7.6]{Wa} for the precise definition of operad actions on topological spaces.) Given a tree $T$ with $|\text{In}(T)|=j,$ we take a $j$-tuple $(g_1,\cdots,g_j)$ of metrics in $\rdok(S^n)$. For each $g_i$, we form the Riemannian manifold $\hat{\sigma}(g_i,e)$ as defined in Section \ref{Hspacesection}. The basepoint of this manifold lies in the middle of a hemisphere of radius 100. We remove this hemisphere, and also the southern (radius 100) hemisphere corresponding to the $i^{th}$ input edge of $T$. We can now smoothly glue these two boundary components together. Repeating this procedure for each $i=1,\cdots,j$ creates a closed Riemannian manifold we will denote $\hat{\sigma}(T;g_1,\cdots,g_j).$ Performing connected sum contractions from the bottom (i.e. starting with the $g_i$) to the top, then yields a metric $\sigma(T;g_1,\cdots,g_j)$ on $S^n$. This will be the metric which results from the tree $T$ acting on $(g_1,\cdots,g_j)$. It is easily checked that the collection (over $j$) of the resulting maps $$W\D(j) \times \rdok(S^n)^j \to \rdok(S^n)$$ indeed form an operad action, and in particular satisfy the associativity requirement stated in \cite[7.6(2)]{Wa}. To aid the reader we provide in Fig. \ref{operadaction} a schematic depiction in the case when $j=4$.

\begin{figure}[!htbp]
\vspace{9.0cm}
\hspace{-2.0cm}
\begin{picture}(0,0)
\includegraphics{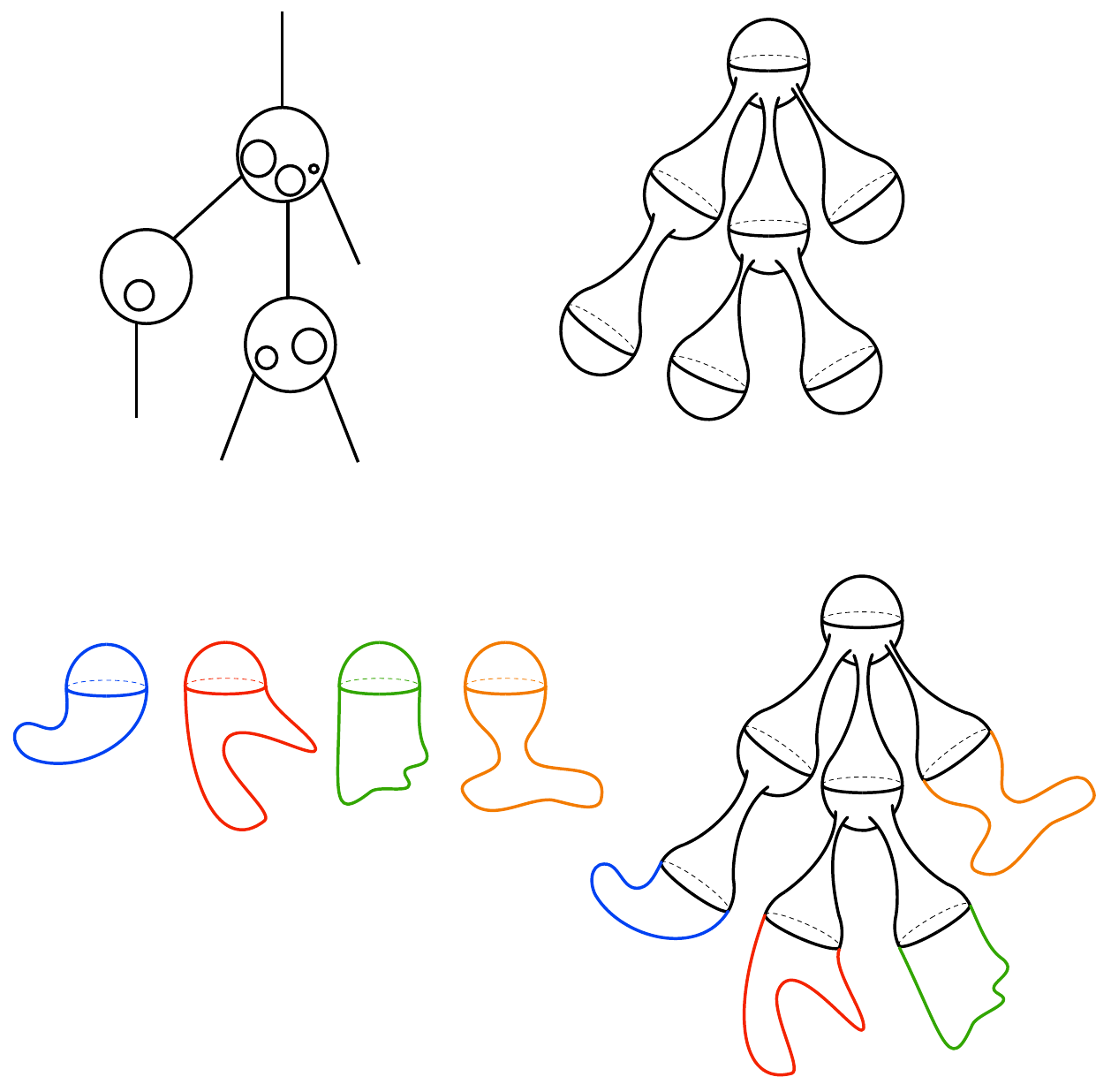}%
\end{picture}
\setlength{\unitlength}{3947sp}%
\begingroup\makeatletter\ifx\SetFigFont\undefined%
\gdef\SetFigFont#1#2#3#4#5{%
  \reset@font\fontsize{#1}{#2pt}%
  \fontfamily{#3}\fontseries{#4}\fontshape{#5}%
  \selectfont}%
\fi\endgroup%
\begin{picture}(5079,1559)(1902,-7227)
\put(2800,-1856){\makebox(0,0)[lb]{\smash{{\SetFigFont{10}{8}{\rmdefault}{\mddefault}{\updefault}{\color[rgb]{0,0,0}$T$}%
}}}}
\put(5200,-1856){\makebox(0,0)[lb]{\smash{{\SetFigFont{10}{8}{\rmdefault}{\mddefault}{\updefault}{\color[rgb]{0,0,0}$\hat{\sigma}(T)$}%
}}}}
\put(2400,-4656){\makebox(0,0)[lb]{\smash{{\SetFigFont{10}{8}{\rmdefault}{\mddefault}{\updefault}{\color[rgb]{0,0,0}$g_1$}%
}}}}
\put(3050,-4656){\makebox(0,0)[lb]{\smash{{\SetFigFont{10}{8}{\rmdefault}{\mddefault}{\updefault}{\color[rgb]{0,0,0}$g_2$}%
}}}}
\put(3800,-4656){\makebox(0,0)[lb]{\smash{{\SetFigFont{10}{8}{\rmdefault}{\mddefault}{\updefault}{\color[rgb]{0,0,0}$g_3$}%
}}}}
\put(4500,-4656){\makebox(0,0)[lb]{\smash{{\SetFigFont{10}{8}{\rmdefault}{\mddefault}{\updefault}{\color[rgb]{0,0,0}$g_4$}%
}}}}
\put(5700,-4300){\makebox(0,0)[lb]{\smash{{\SetFigFont{10}{8}{\rmdefault}{\mddefault}{\updefault}{\color[rgb]{0,0,0}$\hat{\sigma}(T;g_1,g_2,g_3,g_4)$}%
}}}}
\end{picture}%
\caption{}
\label{operadaction}
\end{figure}

In summary, we have now established the following analogue of Lemma 8.2 in \cite{Wa}.
\begin{lemma}\label{WDaction}
When $n\geq 3$, $k\geq 2$, the space $\rdok(S^{n})$ is a $W\D_{n}$-space.
\end{lemma}

\begin{proof}[Proof of Theorem C] 
Simply combine Lemma \ref{WDaction} with Theorems \ref{BVthm} and \ref{BVM}. This shows that the path component of $\rdok(S^n)$ containing the round metric is weakly homotopy equivalent to an $n$-fold loop space. Now use the homotopy equivalence between $\rdok(S^n)$ and ${\mathcal R}^{Ric_k>0}(S^n)$ to deduce the analogous conclusion for ${\mathcal R}^{Ric_k>0}(S^n)$.
\end{proof}  

\subsection{The Grouplike Condition}\label{grouplike}
A natural question to ask is whether Theorem C can be strengthened to hold for the entire space ${\mathcal R}^{Ric_k>0}(S^n)$, and not just the path component containing the round metric. The recognition principle, Theorem \ref{BVM} above, has a stronger form which allows us to recognise when a $\D_{n}$-space, $Z$, is a loop space when $Z$ is not path connected. To understand this we need a definition. Given an $H$-space, $Z$, with a multiplication $\sigma$, we can induce a multiplication $\bar{\sigma}$ on $\pi_{0}(Z)$ in the obvious way:
$$\bar{\sigma}([x],[y])=[\sigma(x,y)],$$
where $x,y\in Z$ and $[x], [y]\in\pi_{0}(Z)$ denote their respective path components. That this is well-defined is an easy exercise. The $H$-space $Z$ is said to be {\em grouplike} if $\pi_{0}(Z)$ forms a group under this multiplication. Note that a $\mathcal{P}$-space, $Z$, (where $\mathcal{P}$ is an operad), is also an $H$-space under the multiplication obtained by restricting the action to any element $c\in \mathcal{P}(2)$. In other words, given such a $c\in\mathcal{P}(2)$ and $x,y\in Z$ we define $$\sigma(x,y)=c.(x,y),$$
where the latter is the operad action on the pair $(x,y)$. Thus, we say that a a $\mathcal{P}$-space $Z$ is {\em grouplike} if this multiplication induces a group multiplication on $\pi_{0}(Z)$. The stronger version of Theorem \ref{BVM} is as follows.
\begin{theorem}[Boardman and Vogt \cite{BV}, May \cite{May}]\label{BVM2}
For any $n\in \mathbb{N}$, a grouplike ${\D}_n$-space, $Z$, is weakly homotopy equivalent to an $n$-fold loop space.
\end{theorem}

In our case, the multiplication induced on $\pi_{0}(\rdk(S^{n}))$ is exactly the same whether we use the original $H$-space multiplication defined in Section \ref{Hspacesection} or the operad action above. The question of whether this multiplication is grouplike is still open and, it seems, hard: the significant challenge is satisfying the inverse axiom. This is discussed somewhat in \cite[Sec 9]{Wa} with regard to the scalar curvature $(k=n)$, and is related to the very difficult problem of deciding whether concordant metrics of positive scalar curvature are isotopic.


\end{document}